\documentclass[a4paper,11pt,fleqn]{article}
\usepackage{pstricks}
\usepackage{enumitem}
\usepackage{amsmath}
\usepackage{amssymb}
\usepackage{pifont}
\usepackage{theorem} 
\usepackage{euscript}
\usepackage{exscale,relsize}
\usepackage{epic,eepic}
\usepackage{multicol}
\usepackage{makeidx}
\usepackage{srcltx}         %for forward references in xdvi
\usepackage{xcolor}
\usepackage{url}
\usepackage{charter}
\usepackage{mathrsfs} %pour mathscr
\usepackage{graphicx}
\usepackage{authblk}
\newcommand{\email}[1]{\href{mailto:#1}{\nolinkurl{#1}}}
\usepackage[normalem]{ulem}     % \sout \xout \uline \uuline
\newlength{\mySubFigSize}
\definecolor{labelkey}{rgb}{0,0.08,0.45}
\definecolor{refkey}{rgb}{0,0.6,0.0}
\definecolor{Brown}{rgb}{0.45,0.0,0.05}
\definecolor{dgreen}{rgb}{0.00,0.49,0.00}
\definecolor{dblue}{rgb}{0,0.08,0.75}
\RequirePackage[dvips,colorlinks,hyperindex]{hyperref}
\hypersetup{linktocpage=true,citecolor=dblue,linkcolor=dgreen}
\renewcommand{\leq}{\ensuremath{\leqslant}}
\renewcommand{\geq}{\ensuremath{\geqslant}}

\newcommand{\minimize}[2]{\ensuremath{\underset{\substack{{#1}}}%
{\text{minimize}}\;\;#2 }}
\newcommand{\Frac}[2]{\displaystyle{\frac{#1}{#2}}} 
 
\newcommand{\scal}[2]{{\langle{{#1}\mid{#2}}\rangle}}

\newcommand{\menge}[2]{\big\{{#1}~\big |~{#2}\big\}}

\newcommand{\HH}{\ensuremath{{\mathcal H}}}
\newcommand{\GG}{\ensuremath{{\mathcal G}}}
\newcommand{\KK}{\ensuremath{{\mathcal K}}}

\newcommand{\emp}{\ensuremath{{\varnothing}}}
\newcommand{\lev}[1]{{\ensuremath{{{{\operatorname{lev}}}_{\leq #1}}\,}}}
\newcommand{\Id}{\ensuremath{\operatorname{Id}}\,}

\newcommand{\RR}{\ensuremath{\mathbb{R}}}
\newcommand{\RP}{\ensuremath{\left[0,+\infty\right[}}
\newcommand{\RM}{\ensuremath{\left]-\infty,0\right]}}

\newcommand{\RPP}{\ensuremath{\left]0,+\infty\right[}}

\newcommand{\RPX}{\ensuremath{\left[0,+\infty\right]}}
\newcommand{\RX}{\ensuremath{\left]-\infty,+\infty\right]}}
\newcommand{\RXX}{\ensuremath{\left[-\infty,+\infty\right]}}

\newcommand{\NN}{\ensuremath{\mathbb N}}

\newcommand{\exi}{\ensuremath{\exists\,}}

\newcommand{\pinf}{\ensuremath{{+\infty}}}
\newcommand{\minf}{\ensuremath{{-\infty}}}

\newcommand{\dom}{\ensuremath{\text{\rm dom}\,}}

\newcommand{\rec}{\ensuremath{\text{\rm rec}\,}}
\newcommand{\prox}{\ensuremath{\text{\rm prox}}}

\newcommand{\sign}{\ensuremath{\text{\rm sign}}}

\newcommand{\cone}{\ensuremath{\text{\rm cone}\,}}

\newcommand{\infconv}{\ensuremath{\mbox{\small$\,\square\,$}}}

\newtheorem{theorem}{Theorem}[section]
\newtheorem{lemma}[theorem]{Lemma}
\newtheorem{corollary}[theorem]{Corollary}
\newtheorem{proposition}[theorem]{Proposition}

\theoremstyle{plain}{\theorembodyfont{\rmfamily}%
}
\theoremstyle{plain}{\theorembodyfont{\rmfamily}%
\newtheorem{example}[theorem]{Example}}
\theoremstyle{plain}{\theorembodyfont{\rmfamily}%
\newtheorem{remark}[theorem]{Remark}}
\theoremstyle{plain}{\theorembodyfont{\rmfamily}%
}
\theoremstyle{plain}{\theorembodyfont{\rmfamily}%
}
\theoremstyle{plain}{\theorembodyfont{\rmfamily}%
}
\theoremstyle{plain}{\theorembodyfont{\rmfamily}%
\newtheorem{definition}[theorem]{Definition}}
\theoremstyle{plain}{\theorembodyfont{\rmfamily}%
}

\numberwithin{equation}{section}
\begin{document}

\title{\sffamily\LARGE Perspective Functions: 
Proximal Calculus and Applications in High-Dimensional 
Statistics\footnote{Contact author: 
P. L. Combettes, \email{plc@math.ncsu.edu},
phone: +1 (919) 515 2671.}}

\author{Patrick L. Combettes$^1$ and Christian L. M\"uller$^2$\\
\small
\small $\!^1$North Carolina State University\\
\small Department of Mathematics\\
\small Raleigh, NC 27695-8205, USA\\
\small \email{plc@math.ncsu.edu}\\
\small \vskip 2mm
\small $\!^2$Flatiron Institute\\
\small Simons Foundation\\
\small New York, NY 10010, USA\\
\small \email{cmueller@simonsfoundation.org}
}

\date{~}

\maketitle

\vskip 8mm

\begin{abstract}
Perspective functions arise explicitly or implicitly in various
forms in applied mathematics and in statistical data analysis. 
To date, no systematic strategy is available to solve the
associated, typically nonsmooth, optimization problems. In this
paper, we fill this gap by showing that proximal methods provide 
an efficient framework to model and solve problems involving
perspective functions. We study the construction of the proximity
operator of a perspective function under general assumptions and
present important instances in which the proximity operator can be
computed explicitly or via straightforward numerical operations.
These results constitute central building blocks in the design of
proximal optimization algorithms. We showcase the versatility of
the framework by designing novel proximal algorithms for
state-of-the-art regression and variable selection schemes in
high-dimensional statistics. 
\end{abstract}

\newpage
\section{Introduction}
Perspective functions appear, often implicitly, in various problems
in areas as diverse as statistics, control, computer vision,
mechanics, game theory,
information theory, signal recovery, transportation theory, machine
learning, disjunctive optimization, and physics (see the companion
paper \cite{Scda16a} for a detailed account). In the setting of a
real Hilbert space $\GG$, the most useful form of a perspective
function, first investigated in Euclidean spaces in \cite{Rock70},
is the following.

\begin{definition}
\label{d:perspective}
Let $\varphi\colon\GG\to\RX$ be a proper lower semicontinuous
convex function and let $\rec\varphi$ be its 
recession function. The perspective of $\varphi$ is
\begin{equation}
\label{eevU74gdnx-08l}
\widetilde{\varphi}\colon\RR\times\GG\to\RX\colon 
(\eta,y)\mapsto
\begin{cases}
\eta\varphi(y/\eta),&\text{if}\;\;\eta>0;\\
(\rec\varphi)(y),&\text{if}\;\;\eta=0;\\
\pinf,&\text{if}\;\;\eta<0.
\end{cases}
\end{equation}
\end{definition}

Many scientific problems result in minimization problems that
involve perspective functions.  In statistics, a prominent instance
is the modeling of data via ``maximum likelihood-type" estimation
(or M-estimation) with a so-called concomitant parameter
\cite{Huber1981}. In this context, $\varphi$ is a likelihood
function, $\eta$ takes the role of the concomitant parameter, e.g.,
an unknown scale or location of the assumed parametric
distribution, and $y$ comprises unknown regression coefficients.
The statistical problem is then to simultaneously estimate the
concomitant variable and the regression vector from data via
optimization. Another important example in statistics 
\cite{Fish25}, signal recovery \cite{Borw96}, and physics 
\cite{Frie07} is the Fisher information of a function 
$x\colon\RR^N\to\RPP$, namely
\begin{equation}
\label{e:fisher}
\int_{\RR^N}\frac{\|\nabla x(t)\|_2^2}{x(t)}dt,
\end{equation}
which hinges on the perspective function of the squared Euclidean
norm (see \cite{Scda16a} for further discussion). 

In the literature, problems involving perspective functions are 
typically solved with a wide range of ad-hoc methods. Despite the
ubiquity of perspective functions, no systematic structuring 
framework has been available to approach these problems. 
The goal of this paper is to fill this gap by showing that
they are amenable to solution by proximal methods, which offer a
broad array of splitting algorithms to solve complex nonsmooth
problems with attractive convergence guarantees 
\cite{Livre1,MaPr16,Siop15,Facc03}. The central element in the
successful implementation of a proximal algorithm is the ability to
compute the proximity operator of the functions present in the
optimization problem. We therefore propose a systematic
investigation of proximity operators for perspective functions
and show that the proximal
framework can efficiently solve perspective-function based
problems, unveiling in particular new applications in
high-dimensional statistics. 

In Section~\ref{sec:2}, we introduce basic concepts from convex
analysis and review essential properties of perspective function.
We then study the proximity operator of perspective functions in
Sections~\ref{sec:3}.  We establish a characterization of the
proximity operator and then provide examples of computation for
concrete instances. Section~\ref{sec:4} unveils new applications of
perspective functions in high-dimensional statistics and
demonstrates the flexibility and potency of the proposed framework
to both model and solve complex problems in statistical data
analysis.  

\section{Notation and background}
\label{sec:2}

\subsection{Notation and elements of convex analysis}
Throughout, $\HH$, $\GG$, and $\KK$ are real Hilbert spaces and  
$\HH\oplus\GG$ denotes their Hilbert direct sum. The symbol
$\|\cdot\|$ denotes the norm of a Hilbert space and 
$\scal{\cdot}{\cdot}$ the associated scalar product.
The closed ball with center $x\in\KK$ and radius $\rho\in\RPP$ is
denoted by $B(x;\rho)$.

A function $f\colon\KK\to\RX$ is proper if 
$\dom f=\menge{x\in\KK}{f(x)<\pinf}\neq\emp$,
coercive if $\lim_{\|x\|\to\pinf}f(x)=\pinf$, and supercoercive
if $\lim_{\|x\|\to\pinf}f(x)/\|x\|=\pinf$.
Denote by $\Gamma_0(\KK)$ the class of proper lower 
semicontinuous convex functions from $\KK$ to $\RX$, and let 
$f\in\Gamma_0(\KK)$. The conjugate of $f$ is the function 
\begin{equation}
\label{e:conj}
f^*\colon\KK\to\RXX\colon u\mapsto\bigg(\sup_{x\in\KK}
\scal{x}{u}-f(x)\bigg).
\end{equation}
It also belongs to $\Gamma_0(\KK)$ and $f^{**}=f$.
The subdifferential of $f$ is the set-valued operator
\begin{equation}
\label{e:subdiff}
\partial f\colon\KK\to 2^\KK\colon x\mapsto 
\menge{u\in\KK}{(\forall y\in\dom f)\;\:
\scal{y-x}{u}+f(x)\leq f(y)}.
\end{equation}
We have
\begin{equation}
\label{e:subdiff2}
(\forall x\in\KK)(\forall u\in\KK)\quad
u\in\partial f(x)\quad\Leftrightarrow\quad
x\in\partial f^*(u).
\end{equation}
Moreover,
\begin{equation}
\label{e:subdiff11}
(\forall x\in\KK)(\forall u\in\KK)\quad
f(x)+f^*(u)\geq\scal{x}{u}
\end{equation}
and 
\begin{equation}
\label{e:subdiff12}
(\forall x\in\KK)(\forall u\in\KK)\quad
u\in\partial f(x)\quad\Leftrightarrow\quad
f(x)+f^*(u)=\scal{x}{u}.
\end{equation}
If $f$ is G\^ateaux differentiable at 
$x\in\dom f$ with gradient $\nabla f(x)$, then 
\begin{equation}
\label{e:2384572k}
\partial f(x)=\{\nabla f(x)\}. 
\end{equation}
Let $z\in\dom f$. The recession function of $f$ is
\begin{equation}
\label{e4RJ3hjhhjeTr6p9-23a}
(\forall y\in\KK)\quad(\rec f)(y)=
\sup_{x\in\dom f}\big(f(x+y)-f(y)\big)=\lim_{\alpha\to\pinf}
\frac{f(z+\alpha y)}{\alpha}.
\end{equation}
The infimal convolution operation is denoted by $\infconv$.
Now let $C$ be a subset of $\KK$. Then 
\begin{equation}
\label{e:iota}
\iota_C\colon\KK\to\{0,\pinf\}\colon x\mapsto
\begin{cases}
0,&\text{if}\;\;x\in C;\\
\pinf,&\text{if}\;\;x\notin C
\end{cases}
\end{equation}
is the indicator function of $C$, 
\begin{equation}
\label{e:dictC}
d_C\colon\KK\to\RPX\colon x\mapsto\inf\|C-x\|
\end{equation}
is the distance function to $C$, and 
\begin{equation}
\label{e:support}
\sigma_C=\iota_C^*\colon\KK\to\RXX\colon u
\mapsto\sup_{x\in C}\scal{x}{u}
\end{equation}
is the support function of $C$. 
If $C$ is nonempty, closed, and convex then, for every 
$x\in\KK$, there exists a unique point $P_Cx\in C$, called 
the projection of $x$ onto $C$, such that $\|x-P_Cx\|=d_C(x)$. 
We have
\begin{equation}
\label{e:kolmogorov}
(\forall x\in\KK)(\forall p\in\KK)\quad
p=P_Cx\quad\Leftrightarrow\quad\big[\,p\in C\quad\text{and}\quad
(\forall y\in C)\quad\scal{y-p}{x-p}\leq 0\,\big].
\end{equation}
The normal cone to $C$ is 
\begin{equation}
\label{e:nc}
N_C=\partial\iota_C\colon\KK\to 2^\KK\colon x\mapsto 
\begin{cases}
\menge{u\in\KK}{\sup\scal{C-x}{u}\leq 0},
&\text{if}\;\;x\in C;\\
\emp,&\text{otherwise.}
\end{cases}
\end{equation}
For further background on convex analysis, see 
\cite{Livre1,Rock70}.

\subsection{Proximity operators}

The proximity operator of $f\in\Gamma_0(\KK)$ is 
\begin{equation}
\label{e:jjm1962}
\prox_f\colon\KK\to\KK\colon x\mapsto
\underset{y\in\KK}{\text{argmin}}
\bigg(f(y)+\frac{1}{2}\|x-y\|^2\bigg).
\end{equation}
This operator was introduced by Moreau in 1962 \cite{Mor62b} to
model problems in unilateral mechanics. In \cite{Smms05}, it was
shown to play an important role in the investigation of various
data processing problems, and it has become increasingly prominent
in the general area of data analysis \cite{Banf11,Wrig12}. We
review basic properties and refer the reader to \cite{Livre1} for a
more complete account. 

Let $f\in\Gamma_0(\KK)$. Then
\begin{equation}
\label{e:jjm}
(\forall x\in\KK)(\forall p\in\KK)\quad
p=\prox_{f}x\quad\Leftrightarrow\quad x-p\in\partial f(p).
\end{equation}
If $C$ is a nonempty closed convex subset of $\KK$, then 
\begin{equation}
\label{e:jjm4}
\prox_{f}=P_C.
\end{equation}
Let $\gamma\in\RPP$. The Moreau decomposition of 
$x\in\KK$ is
\begin{equation}
\label{e:jjm3}
x=\prox_{\gamma f}x+\gamma\prox_{f^*/\gamma}(x/\gamma).
\end{equation}
The following facts will also be needed.

\begin{lemma}
\label{lkIjiiIus3W4-17}
Let $(\Omega,{\EuScript F},\mu)$ be a complete $\sigma$-finite 
measure space, let $\mathsf{K}$ be a separable real Hilbert 
space, and let $\psi\in\Gamma_0({\mathsf K})$. 
Suppose that $\KK=L^2((\Omega,{\EuScript F},\mu);{\mathsf K})$ 
and that $\mu(\Omega)<\pinf$ or $\psi\geq\psi(0)=0$. Set 
\begin{equation}
\label{ekIjiiIus3W4-17A}
\begin{array}{lcll}
\Phi\colon&\KK&\to&\RX\\[2mm]
&x&\mapsto&
\begin{cases}
{\displaystyle\int_\Omega}\psi\big(x(\omega)\big)\mu(d\omega),
&\;\;\text{if}\;\;\psi\circ x\in 
L^1\big((\Omega,{\EuScript F},\mu);\RR\big);\\
\pinf,&\;\;\text{otherwise.}
\end{cases}
\end{array}
\end{equation}
Let $x\in\KK$ and define, for $\mu$-almost every
$\omega\in\Omega$, $p(\omega)=\prox_{\psi} x(\omega)$.
Then $p=\prox_{\Phi}x$.
\end{lemma}
\begin{proof}
By \cite[Proposition~9.32]{Livre1}, $\Phi\in\Gamma_0(\KK)$.
Now take $x$ and $p$ in $\KK$. 
Then it follows from \eqref{e:jjm} and
\cite[Proposition~16.50]{Livre1} that
$p(\omega)=\prox_{\Phi}x(\omega)$\;$\mu$-a.e.
$\Leftrightarrow$ $x(\omega)-p(\omega)\in
\partial\psi(p(\omega))$\;$\mu$-a.e. 
$\Leftrightarrow$ $x-p\in\partial \Phi(p)$.
$\Leftrightarrow$ $p=\prox_{\Phi}x$.
\end{proof}

\begin{lemma}
\label{lQrs54tGhfnnUjH0-02}
Let $D\neq\{0\}$ be a nonempty closed convex subset of $\KK$, 
let $x\in\KK$, and let $\gamma\in\RPP$. 
Set $f=\|\cdot\|+\sigma_D$ and $C=\gamma D$. Then
\begin{equation}
\label{eQrs54tGhfnnUjH0-02g}
\prox_{\gamma f}x=
\begin{cases}
0,&\text{if}\;\;d_{C}(x)\leq\gamma;\\
\bigg(1-\dfrac{\gamma}{d_C(x)}\bigg)
\big(x-P_{C}x\big),&\text{if}\;\;d_{C}(x)>\gamma.
\end{cases}
\end{equation}
If, in addition, $D$ is a cone and $K$ denotes its polar
cone, then $f=\|\cdot\|+\iota_{K}$ and
\begin{equation}
\label{eQrs54tGhfnnUjH0-03g}
\prox_{\gamma f}x=
\begin{cases}
0,&\text{if}\;\;\|P_{K}x\|\leq\gamma;\\
\bigg(1-\dfrac{\gamma}{\|P_{K}x\|}\bigg)
P_{K}x,&\text{if}\;\;\|P_{K}x\|>\gamma.
\end{cases}
\end{equation}
\end{lemma}
\begin{proof}
Using elementary convex analysis, we obtain
\begin{equation}
\label{eQrs54tGhfnnUjH0-02h}
f
=\iota^*_{B(0;1)}+\iota^*_D
=\big(\iota_{B(0;1)}\infconv\iota_D\big)^*
=\iota_{B(0;1)+D}^*
=\sigma_{B(0;1)+D}.
\end{equation}
Hence, it follows from \eqref{e:jjm3} and \eqref{e:jjm4} that
\begin{equation}
\label{eQrs54tGhfnnUjH0-02i}
\prox_{\gamma f}x
=x-\gamma\prox_{f^*/\gamma}(x/\gamma)
=x-\gamma P_{B(0;1)+D}(x/\gamma).
\end{equation}
However by \cite[Propositions~28.1(ii) and~28.10]{Livre1},
\begin{equation}
\label{eQrs54tGhfnnUjH0-02j}
\gamma P_{B(0;1)+D}(x/\gamma)
=P_{B(0;\gamma)+C}\,x=
\begin{cases}
x,&\text{if}\;\;d_{C}(x)\leq\gamma;\\
P_{C}x+\gamma\dfrac{x-P_{C}}{d_{C}(x)},
&\text{if}\;\;d_{C}(x)>\gamma.
\end{cases}
\end{equation}
Upon combining \eqref{eQrs54tGhfnnUjH0-02i} and \eqref{eQrs54tGhfnnUjH0-02j}, we 
arrive at \eqref{eQrs54tGhfnnUjH0-02g}.
Now suppose that, in addition, $D$ is a cone. Then $C=D$,
$\sigma_D=\iota_{K}$, and 
\eqref{e:jjm3} yields $\Id-P_D=P_{K}$. Altogether, 
\eqref{eQrs54tGhfnnUjH0-02g} reduces to \eqref{eQrs54tGhfnnUjH0-03g}.
\end{proof}

\subsection{Perspective functions}

We review here some essential properties of perspective functions.

\begin{lemma}{\rm\cite{Scda16a}}
\label{levU74gdnx-21}
Let $\varphi\in\Gamma_0(\GG)$. Then the following hold:
\begin{enumerate}
\item
\label{levU74gdnx-21ii}
$\widetilde{\varphi}$ is a positively homogeneous function in 
$\Gamma_0(\RR\oplus\GG)$. 
\item
\label{levU74gdnx-21iii}
Let $C=\menge{(\mu,u)\in\RR\times\GG}{\mu+\varphi^*(u)\leq 0}$. 
Then $(\widetilde{\varphi})^*=\iota_C$ and 
$\widetilde{\varphi}=\sigma_C$.
\item
\label{levU74gdnx-21iii+}
Let $\eta\in\RR$ and $y\in\GG$. Then
\begin{equation}
\label{ekIjiiIus3W9-23x}
\partial\widetilde{\varphi}(\eta,y)=
\begin{cases}
\menge{\big(\varphi(y/\eta)-\scal{y}{u}/\eta,u\big)}
{u\in\partial\varphi(y/\eta)},&\text{if}\;\;\eta>0;\\
\menge{(\mu,u)\in C}{\sigma_{\dom\varphi^*}(y)=\scal{u}{y}},
&\text{if}\;\;\eta=0\;\text{and}\;y\neq 0;\\
C,&\text{if}\;\;\eta=0\;\text{and}\;y=0;\\
\emp,&\text{if}\;\;\eta<0.
\end{cases}
\end{equation}
\item
\label{levU74gdnx-21iv}
Suppose that $\dom\varphi^*$ is open or that $\varphi$ is 
supercoercive, let $\eta\in\RR$, and let $y\in\GG$. Then
\begin{equation}
\label{ekIjiiIus3W9-23y}
\partial\widetilde{\varphi}(\eta,y)=
\begin{cases}
\menge{\big(\varphi(y/\eta)-\scal{y}{u}/\eta,u\big)}
{u\in\partial\varphi(y/\eta)},&\text{if}\;\;\eta>0;\\
C,&\text{if}\;\;\eta=0\;\text{and}\;y=0;\\
\emp,&\text{otherwise.}
\end{cases}
\end{equation}
\end{enumerate}
\end{lemma}

We refer to the companion paper \cite{Scda16a} for further
properties of perspective functions as well as examples.
Here are two important instances of (composite) perspective
functions that will play a central role in Section~\ref{sec:4}.  

\begin{lemma}
\label{l:5th}
Let $L\colon\HH\to\GG$ be linear and bounded, let $r\in\GG$, 
let $u\in\HH$, let $\alpha\in\RPP$, let $\rho\in\RR$, and let 
$q\in\left]1,\pinf\right[$. Set 
\begin{equation}
\label{ekIjiiIus3W9-24t}
f\colon\HH\to\RX\colon x\mapsto
\begin{cases}
\Frac{\|Lx-r\|^q}{\alpha|\scal{x}{u}-\rho|^{q-1}},
&\text{if}\;\;\scal{x}{u}>\rho;\\[3mm]
0,&\text{if}\;\;Lx=r\;\;\text{and}\;\;\scal{x}{u}=\rho;\\
\pinf,&\text{otherwise}
\end{cases}
\end{equation}
and $A\colon\HH\to\RR\oplus\GG\colon x\mapsto
(\scal{x}{u}-\rho,Lx-r)$. 
Then $f=[\,\|\cdot\|^q/\alpha]^\sim\circ A\in\Gamma_0(\HH)$. 
\end{lemma}
\begin{proof}
This is a special case of \cite[Example~4.2]{Scda16a}.
\end{proof}

\begin{lemma}{\rm\cite[Example~3.6]{Scda16a}}
\label{levU74gdnx-27}
Let $\phi\in\Gamma_0(\RR)$ be an even function, let $v\in\GG$, 
let $\delta\in\RR$, and set 
\begin{equation}
\label{eevU74gdnx-17A}
g\colon\RR\oplus\GG\to\RX\colon (\eta,y)\mapsto
\begin{cases}
\eta\phi(\|y\|/\eta)+\scal{y}{v}+\delta\eta,&\text{if}\;\;\eta>0;\\
(\rec\phi)(\|y\|)+\scal{y}{v},
&\text{if}\;\;\eta=0;\\
\pinf,&\text{if}\;\;\eta<0.
\end{cases}
\end{equation}
Then $g=[\phi\circ\|\cdot\|+\delta\scal{\cdot}{v}]^\sim
\in\Gamma_0(\RR\oplus\GG)$.
\end{lemma}

\section{Proximity operator of a perspective function}
\label{sec:3}

\subsection{Main result}
We start with a characterization of the proximity operator
of a perspective function when $\dom\varphi^*$ is open.

\begin{theorem}
\label{tevU74gdnx-06}
Let $\varphi\in\Gamma_0(\GG)$, let $\gamma\in\RPP$, let
$\eta\in\RR$, and let $y\in\GG$. Then the following hold:
\begin{enumerate}
\item
\label{tevU74gdnx-06i}
Suppose that $\eta+\gamma\varphi^*(y/\gamma)\leq 0$. 
Then $\prox_{\gamma\widetilde{\varphi}}(\eta,y)=(0,0)$.
\item
\label{tevU74gdnx-06ii}
Suppose that $\dom\varphi^*$ is open and that 
$\eta+\gamma\varphi^*(y/\gamma)>0$. Then 
\begin{equation}
\label{e:DC2016-03-01a}
\prox_{\gamma\widetilde{\varphi}}(\eta,y)=
\big(\eta+\gamma\varphi^*(p),y-\gamma p\big), 
\end{equation}
where $p$ is the unique solution to the inclusion
\begin{equation}
\label{e:scdaJune2016}
y\in\gamma p+\big(\eta+\gamma\varphi^*(p)\big)\partial\varphi^*(p).
\end{equation}
If $\varphi^*$ is differentiable at $p$, then
$p$ is characterized by
$y=\gamma p+(\eta+\gamma\varphi^*(p))\nabla\varphi^*(p)$.
\end{enumerate}
\end{theorem}
\begin{proof}
It follows from 
Lemma~\ref{levU74gdnx-21}\ref{levU74gdnx-21iii} that
\begin{equation}
\label{eqerY76t0k3-10a}
\widetilde{\varphi}=\sigma_C,\quad\text{where}\quad
C=\menge{(\mu,u)\in\RR\oplus\GG}{\mu+\varphi^*(u)\leq 0}.
\end{equation}
Since $\varphi\in\Gamma_0(\GG)$, we have 
$\varphi^*\in\Gamma_0(\GG)$. Therefore, 
$C$ is a nonempty closed convex set. In turn, we derive
from \cite[Proposition~3.2]{Siop07} that 
$\prox_{\gamma\widetilde{\varphi}}=\prox_{\sigma_{\gamma C}}$ 
is a proximal thresholder on $\gamma C$ in the sense that
\begin{equation}
\label{e:oulan}
(\forall\eta\in\RR)(\forall y\in\GG)\quad
\prox_{\gamma\widetilde{\varphi}}(\eta,y)=(0,0)
\quad\Leftrightarrow\quad(\eta,y)\in\gamma C.
\end{equation}

\ref{tevU74gdnx-06i}: 
By \eqref{eqerY76t0k3-10a} and \eqref{e:oulan},
$(\forall\eta\in\RR)(\forall y\in\GG)$ 
$\prox_{\gamma\widetilde{\varphi}}(\eta,y)=(0,0)$
$\Leftrightarrow$ $\eta+\gamma\varphi^*(y/\gamma)\leq 0$.

\ref{tevU74gdnx-06ii}: Set
$(\chi,q)=\prox_{\gamma\widetilde{\varphi}}(\eta,y)$ and
$p=(y-q)/\gamma$. It follows from \eqref{e:jjm} that 
$(\chi,q)\in\dom(\gamma\partial\widetilde{\varphi})$ and from
\eqref{e:oulan} that $(\chi,q)\neq(0,0)$. Hence, we deduce from 
Lemma~\ref{levU74gdnx-21}\ref{levU74gdnx-21iv} that
$\chi>0$. Furthermore, we derive from \eqref{e:jjm} and 
Lemma~\ref{levU74gdnx-21}\ref{levU74gdnx-21iii+} that
$(\chi,q)$ is characterized by
\begin{equation}
\label{ekIjiiIus3W9-16a}
\eta-\chi=\gamma\varphi(q/\chi)-\scal{q/\chi}{y-q}
\quad\text{and}\quad y-q\in\gamma\partial\varphi(q/\chi),
\end{equation}
i.e., 
\begin{equation}
\label{ekIjiiIus3W9-17a}
(\eta-\chi)/\gamma=\varphi(q/\chi)-\scal{q/\chi}{p}
\quad\text{and}\quad p\in\partial\varphi(q/\chi).
\end{equation}
However, \eqref{e:subdiff12} asserts that 
\begin{equation}
\label{ekIjiiIus3W9-16b}
p\in\partial\varphi(q/\chi)\quad\Leftrightarrow\quad
\varphi(q/\chi)+\varphi^*(p)=\scal{q/\chi}{p}.
\end{equation}
Hence, we derive from \eqref{ekIjiiIus3W9-17a} that
$\varphi^*(p)=(\chi-\eta)/\gamma$, i.e.,
\begin{equation}
\label{ekIjiiIus3W9-17b}
\chi=\eta+\gamma\varphi^*(p).
\end{equation}
Hence, by \eqref{e:subdiff2}, 
\begin{equation}
\label{ekIjiiIus3W9-17c}
p\in\partial\varphi(q/\chi)\quad\Leftrightarrow\quad
q\in\chi\partial\varphi^*(p)\quad\Leftrightarrow\quad
y\in\gamma p+\big(\eta+\gamma\varphi^*(p)\big)\partial\varphi^*(p).
\end{equation}
Altogether, we have established the characterization
\eqref{e:DC2016-03-01a}--\eqref{e:scdaJune2016}, while the
assertion concerning the differentiable case follows from 
\eqref{e:2384572k}.
\end{proof}

\begin{remark}
\label{rkIjiiIus3W6-17}
Here is an alternative proof of Theorem~\ref{tevU74gdnx-06}.
It follows from 
Lemma~\ref{levU74gdnx-21}\ref{levU74gdnx-21iii} that
\begin{equation}
\label{eqerY76t0k3-10A}
\big(\widetilde{\varphi}\big)^*=\iota_C,\quad\text{where}\quad
C=\menge{(\mu,u)\in\RR\oplus\GG}{\mu+\varphi^*(u)\leq 0}
\end{equation}
is a nonempty closed convex set. Hence, using
\eqref{e:jjm3} and \eqref{e:jjm4}, we obtain
\begin{equation}
\label{eqerY76t0k3-10b}
\prox_{\gamma \widetilde{\varphi}}(\eta,y)
=(\eta,y)-\gamma\prox_{\gamma^{-1}\widetilde{\varphi}^*}
\big(\eta/\gamma,y/\gamma\big)
=(\eta,y)-\gamma P_C\big(\eta/\gamma,y/\gamma\big)
=(\eta,y)-P_{\gamma C}\big(\eta,y\big).
\end{equation}
Now set $(\pi,p)=P_C(\eta/\gamma,y/\gamma)$. We deduce from
\eqref{e:jjm4}, \eqref{e:jjm3}, and \eqref{e:nc} that 
$(\pi,p)$ is characterized by 
\begin{equation}
\label{eevU74gdnx-12a}
\big(\eta/\gamma-\pi,y/\gamma-p\big)\in N_C(\pi,p). 
\end{equation}

\ref{tevU74gdnx-06i}:
We have $(\eta/\gamma,y/\gamma)\in C$. Hence, 
$(\pi,p)=(\eta/\gamma,y/\gamma)$ and \eqref{eqerY76t0k3-10b} yields
$\prox_{\gamma \widetilde{\varphi}}(\eta,y)=(0,0)$.

\ref{tevU74gdnx-06ii}: Set
$h\colon\RR\oplus\GG\to\RX\colon(\mu,u)\mapsto\mu+\varphi^*(u)$.
Then $C=\lev{0}h$ and $\dom h=\RR\times\dom\varphi^*$ is open.
It therefore follows from \cite[Proposition~6.43(ii)]{Livre1} that
\begin{equation}
\label{ekIjiiIus3W9-15}
N_{\dom h}(\pi,p)=\{(0,0)\}.
\end{equation}
Now let $z\in\dom\varphi^*$ and let
$\zeta\in\left]\minf,-\varphi^*(z)\right[$. Then
$h(\zeta,z)<0$. Therefore, we derive from
\cite[Lemma~26.17 and Proposition~16.8]{Livre1} and 
\eqref{ekIjiiIus3W9-15} that
\begin{align}
N_C(\pi,p)
&=
\begin{cases}
N_{\dom h}(\pi,p)\cup\cone\,\partial h(\pi,p),
&\text{if}\;\;\pi+\varphi^*(p)=0;\\
N_{\dom h}(\pi,p),&\text{if}\;\;\pi+\varphi^*(p)<0
\end{cases}\label{e:hj4mB}\\
&=
\begin{cases}
\cone\,\partial h(\pi,p),
&\text{if}\;\;\pi+\varphi^*(p)=0;\\
\{(0,0)\},&\text{if}\;\;\pi+\varphi^*(p)<0
\end{cases}\nonumber\\
&=
\begin{cases}
\cone\big(\{1\}\times\partial\varphi^*(p)\big),
&\text{if}\;\;\pi=-\varphi^*(p);\\
\{(0,0)\},&\text{if}\;\;\pi<-\varphi^*(p).
\label{eevU74gdnx-12b}
\end{cases}
\end{align}
Hence, if $\pi<-\varphi^*(p)$, then \eqref{eevU74gdnx-12a}
yields $(\eta/\gamma-\pi,y/\gamma-p)=(0,0)$ and therefore
$(\eta/\gamma,y/\gamma)=(\pi,p)\in C$, which is impossible since 
$(\eta/\gamma,y/\gamma)\notin C$.
Thus, the characterization \eqref{eevU74gdnx-12a} becomes
\begin{equation}
\label{eevU74gdnx-12c}
\pi=-\varphi^*(p)\quad\text{and}\quad
(\exi\nu\in\RPP)(\exi w\in\partial\varphi^*(p))\quad
\big(\eta/\gamma+\varphi^*(p),y/\gamma-p\big)=\nu(1,w)
\end{equation}
that is,
$y\in\gamma p+(\eta+\gamma\varphi^*(p))\partial\varphi^*(p)$.
\end{remark}

\begin{remark}
\label{rkIjiiIus3W9-17}
Let $\varphi\in\Gamma_0(\GG)$ be such that $\dom\varphi^*$ is
open, let $\gamma\in\RPP$, let
$\eta\in\RR$, and let $y\in\GG$ be such that
$\eta+\gamma\varphi^*(y/\gamma)>0$. We derive from 
\eqref{ekIjiiIus3W9-16a} that
$y/\chi-q/\chi\in\partial(\gamma\varphi/\chi)(q/\chi)$ and then
from \eqref{e:jjm} that $q=\chi\prox_{\gamma\varphi/\chi}(y/\chi)$.
Using \eqref{e:jjm3}, we can
also write $q=y-\prox_{\chi\gamma\varphi^*(\cdot/\gamma)}y$.
Hence, we deduce from Theorem~\ref{tevU74gdnx-06} the implicit
relation
\begin{equation}
\label{ekIjiiIus3W9-17x}
\prox_{\gamma\widetilde{\varphi}}(\eta,y)=
\chi\left(1,\prox_{\gamma\varphi/\chi}(y/\chi)\right),
\quad\text{where}\quad\chi=\eta+\gamma\varphi^*
\bigg(\frac{\prox_{\chi\gamma\varphi^*(\cdot/\gamma)}y}{\gamma}
\bigg).
\end{equation}
\end{remark}

The next example is based on distance functions.

\begin{example}
\label{exevU74gdnx-17a}
Let $\varphi=\phi\circ d_D$, where $D=B(0;1)\subset\GG$ and 
$\phi\in\Gamma_0(\RR)$ 
is an even function such that $\phi(0)=0$ and
$\phi^*$ is differentiable on $\RR$. It follows from  
\cite[Examples~13.3(iv)~and~13.23]{Livre1} that
$\varphi^*=\|\cdot\|+\phi^*\circ\|\cdot\|$. Note that, since
$\varphi$ and $\phi$ are even and satisfy 
$\varphi(0)=0$ and $\phi(0)=0$, 
$\varphi^*$ and $\phi^*$ are even and satisfy 
$\varphi^*(0)=0$ and $\phi^*(0)=0$ as well by
\cite[Propositions~13.18 and 13.19]{Livre1}.
In turn, $\phi^{*'}(0)=0$ and we therefore derive from
\cite[Corollary~16.38(iii) and Example~16.25]{Livre1} that
\begin{equation}
\label{e:subnorm}
(\forall u\in\GG)\quad\partial\varphi^*(u)=
\begin{cases}
\bigg\{\dfrac{1+\phi^{*'}(\|u\|)}{\|u\|}u\bigg\},
&\text{if}\;\;u\neq 0;\\
B(0;1),&\text{if}\;\;u=0.
\end{cases}
\end{equation}
We have $\dom\varphi^*=\GG$ and, in view of 
Theorem~\ref{tevU74gdnx-06}\ref{tevU74gdnx-06ii},
we need only assume that 
$\eta+\gamma\varphi^*(y/\gamma)>0$, i.e.,
\begin{equation}
\label{eevU74gdnx-17a}
\eta+\|y\|+\gamma\phi^*(\|y\|/\gamma)>0. 
\end{equation}
Then \eqref{e:scdaJune2016} and \eqref{e:subnorm} yield
\begin{equation}
\label{eevU74gdnx-17b}
\begin{cases}
y=\gamma p+\Big(\eta+\gamma\big(\|p\|+\phi^*(\|p\|)\big)\Big)
\dfrac{1+\phi^{*'}(\|p\|)}{\|p\|}p,&\text{if}\;\;p\neq 0;\\
\|y\|\leq\eta,&\text{if}\;\;p=0.
\end{cases}
\end{equation}
In view of Remark~\ref{rkIjiiIus3W6-17}, the normal cone to 
the set $C$ of \eqref{eqerY76t0k3-10A} at $(0,0)$ is
\begin{equation}
K=\menge{(\eta,y)\in\RP\times\GG}{\|y\|\leq\eta}. 
\end{equation}
So, 
for every $(\eta,y)\in K$, $P_C(\eta/\gamma,y/\gamma)=(0,0)$
and $\prox_{\gamma\widetilde{\varphi}}(\eta,y)=(\eta,y)$. 
Now suppose that $(\eta,y)\notin K$. Then $p\neq 0$ and, taking the 
norm in the upper line of \eqref{eevU74gdnx-17b}, we obtain 
\begin{equation}
\label{eevU74gdnx-17c}
\gamma\|p\|+\Big(\eta+\gamma\big(\|p\|+\phi^*(\|p\|)\big)\Big)
\big(1+\phi^{*'}(\|p\|)\big)=\|y\|.
\end{equation}
Set
\begin{equation}
\label{eevU74gdnx-17x}
\psi\colon s\mapsto
s+\bigg(\dfrac{\eta}{\gamma}+s+\phi^*(s)\bigg)
\big(1+\phi^{*'}(s)\big)-\frac{\|y\|}{\gamma}
\end{equation}
and define
\begin{equation}
\label{eevU74gdnx-16x}
\theta\colon s\mapsto\Frac{1}{2}\bigg(
\bigg(\frac{\eta}{\gamma}+s+\phi^*(s)\bigg)^2+s^2\bigg)
-\frac{\|y\|s}{\gamma}.
\end{equation}
Since $\phi^*$ is convex, $\theta$ is strongly convex and it
therefore admits a unique minimizer $t$. Therefore 
$\psi(t)=\theta'(t)=0$ and $\|p\|=t=\psi^{-1}(\|y\|/\gamma)$ 
is the unique solution to \eqref{eevU74gdnx-17c}.
In turn, \eqref{eevU74gdnx-17b} yields
\begin{equation}
p=\frac{t}{\|y\|+\gamma\psi(t)}y,
\end{equation}
and we obtain $\prox_{\gamma\widetilde{\varphi}}(\eta,y)$
via \eqref{e:DC2016-03-01a}.
\end{example}

Next, we compute the proximity operator of a special case of the
perspective function introduced in Lemma~\ref{levU74gdnx-27}.

\begin{corollary}
\label{cevU74gdnx-29a}
Let $v\in\GG$, let $\delta\in\RR$, and let $\phi\in\Gamma_0(\RR)$ 
be an even function such that $\phi(0)=0$ and
$\phi^*$ is differentiable on $\RR$. Define 
\begin{equation}
\label{eevU74gdnx-13a}
g\colon\RR\oplus\GG\to\RX\colon (\eta,y)\mapsto
\begin{cases}
\eta\phi(\|y\|/\eta)+\delta\eta+\scal{y}{v},&\text{if}\;\;\eta>0;\\
0,&\text{if}\;\;y=0\;\;\text{and}\;\;\eta=0;\\
\pinf,&\text{otherwise.}
\end{cases}
\end{equation}
Let $\gamma\in\RPP$, let $\eta\in\RR$, let $y\in\GG$, and set
\begin{equation}
\label{eevU74gdnx-15x}
\psi\colon
s\mapsto\bigg(\phi^*(s)+\frac{\eta}{\gamma}-\delta\bigg)
\phi^{*'}(s)+s.
\end{equation}
Then $\psi$ is invertible. Moreover, if 
$\eta+\gamma\phi^*(\|y/\gamma-v\|)>\gamma\delta$, set 
\begin{equation}
\label{eevU74gdnx-13c}
t=\psi^{-1}\big(\|y/\gamma-v\|\big)
\quad\text{and}\quad p=
\begin{cases}
v+\dfrac{t}{\|y-\gamma v\|}(y-\gamma v),
&\text{if}\;\;y\neq\gamma v;\\
v,&\text{if}\;\;y=\gamma v.
\end{cases}
\end{equation}
Then 
\begin{equation}
\label{eevU74gdnx-13d}
\prox_{\gamma g}(\eta,y)=
\begin{cases}
\big(\eta+\gamma(\phi^*(t)-\delta),y-\gamma p\big),&\text{if}\;\;
\eta+\gamma\phi^*(\|y/\gamma-v\|)>\gamma\delta;\\
(0,0),&\text{if}\;\;
\eta+\gamma\phi^*(\|y/\gamma-v\|)\leq\gamma\delta.
\end{cases}
\end{equation}
\end{corollary}
\begin{proof}
This is a special case of Theorem~\ref{tevU74gdnx-06} with
$\varphi=\phi\circ\|\cdot\|+\delta+\scal{\cdot}{v}$. Indeed, as
shown in \cite[Example~3.6]{Scda16a}, \eqref{eevU74gdnx-13a} is 
a special case of \eqref{eevU74gdnx-17A}. Hence, we 
derive from Lemma~\ref{levU74gdnx-27} that 
$g=\widetilde{\varphi}\in\Gamma_0(\RR\oplus\GG)$.
Next, we obtain from 
\cite[Example~13.7 and Proposition~13.20(iii)]{Livre1} that
\begin{equation}
\label{eevU74gdnx-29c}
\varphi^*=\phi^*\circ\|\cdot-v\|-\delta 
\end{equation}
and therefore that
\begin{equation}
\label{eevU74gdnx-29t}
\nabla\varphi^*\colon\GG\to\GG\colon z\mapsto
\begin{cases}
\dfrac{\phi^{*'}(\|z-v\|)}{\|z-v\|}(z-v),&\text{if}\;\;z\neq v;\\
0,&\text{if}\;\;z=v.
\end{cases}
\end{equation}
In view of Theorem~\ref{tevU74gdnx-06}, it remains to 
assume that $\eta+\gamma\varphi^*(y/\gamma)>0$, i.e., 
$\eta+\phi^*(\|y/\gamma-v\|)>\gamma\delta$, and to show that 
the point $(t,p)$ provided by \eqref{eevU74gdnx-13c} satisfies 
\begin{equation}
\label{eDACy47Hj-10v}
t=\|p-v\|\quad\text{and}\quad
y=\gamma p+\big(\eta+\gamma\varphi^*(p)\big)\nabla\varphi^*(p).
\end{equation}
We consider two cases:
\begin{itemize}
\item
$y=\gamma v$: Since $\phi$ is an even convex function such that
$\phi(0)=0$, $\phi^*$ has the same properties by
\cite[Propositions~13.18~and~13.19]{Livre1}. Hence, going back to
Remark~\ref{rkIjiiIus3W6-17}, since $\phi^*$ is differentiable, 
the points that have $(\pi,p)=(\delta,v)$
as a projection onto 
$C=\menge{(\mu,u)\in\RR\oplus\GG}{\mu+\phi^*(\|u-v\|)\leq\delta}$
are the points on the ray 
$\menge{(\delta+\lambda,v)}{\lambda\in\RP}$. Thus, we derive from 
\eqref{eqerY76t0k3-10b} that
\begin{multline}
\label{eDACy47Hj-10a}
y=\gamma v\;\Leftrightarrow\;
P_C(\eta/\gamma,y/\gamma)=(\pi,p)=(\delta,v)
\;\Leftrightarrow\;\;p=v\Leftrightarrow\;t=0\;\Leftrightarrow\;\\
\prox_{\gamma \widetilde{\varphi}}(\eta,y)
=(\eta,y)-\gamma(\delta,v)=(\eta-\gamma\delta,y-\gamma p).
\end{multline}
Since $\phi^*(0)=0$, we recover \eqref{eevU74gdnx-13d}.
\item
$y\neq\gamma v$: As seen in \eqref{eDACy47Hj-10a}, $p\neq v$.
Using \eqref{eevU74gdnx-29c} and \eqref{eevU74gdnx-29t}, 
\eqref{eDACy47Hj-10v} can be rewritten as 
\begin{equation}
\label{eevU74gdnx-29d}
t=\|p-v\|\quad\text{and}\quad y-\gamma v=
\gamma(p-v)+\frac{\big(\eta+\gamma\phi^*(\|p-v\|)-\gamma\delta\big)
\phi^{*'}(\|p-v\|)}{\|p-v\|}(p-v),
\end{equation}
that is,
\begin{equation}
\label{eevU74gdnx-29z}
t=\|p-v\|\quad\text{and}\quad y/\gamma-v=
\frac{\|p-v\|+\big(\eta/\gamma-\delta+\phi^*(\|p-v\|)\big)
\phi^{*'}(\|p-v\|)}{\|p-v\|}(p-v).
\end{equation}
In view of \eqref{eevU74gdnx-15x}, this is equivalent to
\begin{equation}
\label{eevU74gdnx-30b}
t=\|p-v\|\quad\text{and}\quad y/\gamma-v=
\frac{\psi(\|p-v\|)}{\|p-v\|}(p-v).
\end{equation}
Upon taking the norm on both sides of the second equality, we obtain
\begin{equation}
\label{eevU74gdnx-30a}
\psi(t)=\psi(\|p-v\|)=\|y/\gamma-v\|.
\end{equation}
We note that, since $\phi^*$ is convex, 
$\psi$ is the derivative of the strongly convex function
\begin{equation}
\label{eevU74gdnx-13g}
\theta\colon s\mapsto\frac{1}{2}\Big(\phi^{*2}(s)+s^2\Big)+
\bigg(\frac{\eta}{\gamma}-\delta\bigg)\phi^*(s).
\end{equation}
Consequently, $\psi$ is strictly increasing
\cite[Proposition~17.13]{Livre1}, hence invertible.
It follows that $t=\psi^{-1}(\|y/\gamma-v\|)$.
In turn, \eqref{eevU74gdnx-30b} yields 
\eqref{eevU74gdnx-13c}.
\end{itemize}
\end{proof}

\begin{example}
\label{exevU74gdnx-15a}
Define
\begin{equation}
\label{eevU74gdnx-15a}
g\colon\RR\oplus\GG\to\RX\colon (\eta,y)\mapsto
\begin{cases}
-\sqrt{\eta^2-\|y\|^2},&\text{if}\;\;\eta>0\;\;\text{and}\;\;
\|y\|\leq\eta;\\
0,&\text{if}\;\;y=0\;\;\text{and}\;\;\eta=0;\\
\pinf,&\text{otherwise,}
\end{cases}
\end{equation}
let $\gamma\in\RPP$, let $\eta\in\RR$, let $y\in\GG$, and define
\begin{equation}
\label{eevU74gdnx-15y}
\psi\colon s\mapsto\bigg(2+\frac{\eta}{\gamma\sqrt{1+s^2}}\bigg)s.
\end{equation}
If $\eta+\sqrt{\gamma^2+\|y\|^2}>0$, set 
\begin{equation}
\label{eevU74gdnx-15i}
p=
\begin{cases}
\dfrac{t}{\|y\|}y,&\text{if}\;\;y\neq 0;\\
0,&\text{if}\;\;y=0,
\end{cases}
\qquad\text{where}\quad
t=\psi^{-1}\bigg(\frac{\|y\|}{\gamma}\bigg).
\end{equation}
Then 
\begin{equation}
\label{eevU74gdnx-15d}
\prox_{\gamma g}(\eta,y)=
\begin{cases}
\Big(\eta+\gamma\sqrt{1+t^2},y-\gamma p\Big),&\text{if}\;\;
\eta+\sqrt{\gamma^2+\|y\|^2}>0;\\
(0,0),&\text{if}\;\;
\eta+\sqrt{\gamma^2+\|y\|^2}\leq 0.
\end{cases}
\end{equation}
\end{example}
\begin{proof}
This is a special case of Corollary~\ref{cevU74gdnx-29a} with 
with $\delta=0$, $v=0$, and 
\begin{equation}
\label{eevU74gdnx-15b}
\phi\colon s\mapsto
\begin{cases}
-\sqrt{1-s^2},&\text{if}\;\;|s|\leq 1;\\
\pinf,&\text{otherwise.}
\end{cases}
\end{equation}
It follows from \cite[Example~13.2(vi) and Corollary~13.33]{Livre1} 
that $\phi^*\colon s\mapsto\sqrt{1+s^2}$. Hence, 
$\phi^{*'}\colon s\mapsto s/\sqrt{1+s^2}$ and we derive
\eqref{eevU74gdnx-15d} from \eqref{eevU74gdnx-13d}.
\end{proof}

\begin{example}
\label{exevU74gdnx-10a}
Let $v\in\GG$, let $\delta\in\RR$, 
let $\alpha\in\RPP$, let $q\in\left]1,\pinf\right[$, and 
consider the function
\begin{equation}
\label{eevU74gdnx-10a}
g\colon\RR\oplus\GG\to\RX\colon(\eta,y)\mapsto
\begin{cases}
\Frac{\|y\|^q}{\alpha\eta^{q-1}}+\delta\eta+\scal{y}{v},
&\text{if}\;\;\eta>0; \\
0,&\text{if}\;\;y=0\;\text{and}\;\eta=0;\\
\pinf, &\text{otherwise.}
\end{cases}
\end{equation}
Let $\gamma\in\RPP$, set $q^*=q/(q-1)$, set 
$\varrho=(\alpha(1-1/q^*))^{q^*-1}$, and take 
$\eta\in\RR$ and $y\in\GG$. If
$q^*\gamma^{q^*-1}\eta+\varrho\|y\|^{q^*}>\gamma\delta$ and
$y\neq \gamma v$, let $t$ be the unique 
solution in $\RPP$ to the equation
\begin{equation}
\label{eevU74gdnx-10f}
s^{2q^*-1}+\frac{q^*(\eta-\gamma\delta)}{\gamma\varrho}s^{q^*-1}
+\frac{q^*}{\varrho^2}s-\frac{q^*\|y-\gamma v\|}{\gamma\varrho^2}=0
\end{equation}
and set
\begin{equation}
\label{eevU74gdnx-10x}
p=
\begin{cases}
v+\dfrac{t}{\|y-\gamma v\|}(y-\gamma v),
&\text{if}\;\;y\neq\gamma v;\\
v,&\text{if}\;\;y=\gamma v.
\end{cases}
\end{equation}
Then
\begin{equation}
\label{eevU74gdnx-10b}
\prox_{\gamma g}(\eta,y)=
\begin{cases}
\big(\eta+\gamma(\varrho t^{q^*}-\delta)/q^*,y-\gamma p\big),
&\text{if}\;\;q^*\gamma^{q^*-1}\eta+\varrho\|y\|^{q^*}
>\gamma\delta;\\
\big(0,0\big),&\text{if}\;\;q^*\gamma^{q^*-1}\eta+
\varrho\|y\|^{q^*}\leq\gamma\delta.
\end{cases}
\end{equation}
\end{example}
\begin{proof}
This is a special case of Corollary~\ref{cevU74gdnx-29a} with
$\phi=|\cdot|^q/\alpha$. Indeed, we derive from 
\cite[Example~13.2(i) and Proposition~13.20(i)]{Livre1} 
that $\phi^*=\varrho|\cdot|^{q^*}/q^*$, which implies that 
\eqref{eevU74gdnx-10x}--\eqref{eevU74gdnx-10b} follow from 
\eqref{eevU74gdnx-13d}.
\end{proof}

\begin{example}
\label{exevU74gdnx-12a}
Let $v\in\GG$, let $\alpha\in\RPP$, let $\delta\in\RR$, 
and consider the function
\begin{equation}
\label{eevU74gdnx-12x}
g\colon\RR\oplus\GG\to\RX\colon(\eta,y)\mapsto
\begin{cases}
\Frac{\|y\|^2}{\alpha\eta}+\delta\eta+\scal{y}{v},
&\text{if}\;\;\eta>0; \\
0, &\text{if}\;\;y=0\;\text{and}\;\eta=0;\\
\pinf, &\text{otherwise.}
\end{cases}
\end{equation}
We obtain a special case of Example~\ref{exevU74gdnx-10a} with
$q=q^*=2$. Now let $\gamma\in\RPP$, and take 
$\eta\in\RR$ and $y\in\GG$.
If $4\gamma\eta+\alpha\|y\|^2\leq 2\gamma\delta$, then
$\prox_{\gamma g}(\eta,y)=(0,0)$. Suppose that 
$4\gamma\eta+\alpha\|y\|^2>2\gamma\delta$.
First, if $y=\gamma v$, then 
$\prox_{\gamma g}(\eta,y)=(\eta-\gamma\delta/2,0)$. 
Next, suppose that $y\neq\gamma v$ and let $t$ be the unique 
solution in $\RPP$ to the depressed cubic equation 
\begin{equation}
\label{eevU74gdnx-10p}
s^3+\frac{4\alpha(\eta-\gamma\delta)+8\gamma}{\alpha^2\gamma}s
-\frac{8\|y-\gamma v\|}{\alpha^2\gamma}=0.
\end{equation}
Then we derive from 
\eqref{eevU74gdnx-10x}--\eqref{eevU74gdnx-10b} that
\begin{equation}
\label{eDACy47Hj-13b}
\prox_{\gamma g}(\eta,y)=
\bigg(\eta+\dfrac{\gamma}{2}\bigg(\dfrac{\alpha
t^{2}}{2}-\delta\bigg),\bigg(1-
\dfrac{\gamma t}{\|y-\gamma v\|}\bigg)(y-\gamma v)\bigg).
\end{equation}
Note that \eqref{eevU74gdnx-10p} can be solved explicitly via 
Cardano's formula \cite[Chapter~4]{Birk77} to obtain $t$.
\end{example}

We conclude this subsection by investigating integral functions
constructed from integrands that are perspective functions.

\begin{proposition}
\label{pkIjiiIus3W9-19}
Let $(\Omega,{\EuScript F},\mu)$ be a measure space, let
${\mathsf G}$ be a separable real Hilbert space, and let 
$\varphi\in\Gamma_0(\mathsf{G})$. Set
$\HH=L^2((\Omega,{\EuScript F},\mu);\RR)$ and
$\GG=L^2((\Omega,{\EuScript F},\mu);\mathsf{G})$, and suppose
that $\mu(\Omega)<\pinf$ or $\varphi\geq\varphi(0)=0$. 
For every $x\in\HH$, set 
$\Omega_0(x)=\menge{\omega\in\Omega}{x(\omega)=0}$ and
$\Omega_+(x)=\menge{\omega\in\Omega}{x(\omega)>0}$.
Define 
\begin{multline}
\label{e:rodos2016-08-02b}
\Phi\colon\HH\oplus\GG\to\RX\colon (x,y)\mapsto\\
\begin{cases}
\displaystyle{\int_{\Omega_0(x)}}
\big(\rec\varphi\big)\big(y(\omega)\big)\mu(d\omega)\!\!\!&+
\displaystyle{\int_{\Omega_+(x)}}x(\omega)
\varphi\bigg(\dfrac{y(\omega)}{x(\omega)}\bigg)\mu(d\omega),\\[5mm]
&\text{if}\;\;
\begin{cases}
x\geq 0\;\:\mu\text{-a.e.}\\
(\rec\varphi)(y)1_{\Omega_0(x)}+x\varphi(y/x)1_{\Omega_+(x)}\in 
L^1\big((\Omega,\mathcal{F},\mu);\RR\big);\\[3mm]
\end{cases}\\
\pinf,&\text{otherwise.}
\end{cases}
\end{multline}
Now let $x\in\HH$ and $y\in\GG$, and set, for $\mu$-almost every
$\omega\in\Omega$, $(p(\omega),q(\omega))=
\prox_{\widetilde{\varphi}} (x(\omega),y(\omega))$.
Then $\prox_{\Phi} (x,y)=(p,q)$.
\end{proposition}
\begin{proof}
Set $z=(x,y)$. It follows from 
Lemma~\ref{levU74gdnx-21}\ref{levU74gdnx-21ii} 
that $\widetilde{\varphi}\in\Gamma_0(\RR\oplus\mathsf{G})$, 
and \cite[Proposition~5.1]{Scda16a} asserts that $\Phi$ is a 
well-defined function in $\Gamma_0(\RR\oplus\GG)$ with 
\begin{equation}
\label{e:245hg2}
\Phi(z)=\int_{\Omega}\widetilde{\varphi}\big(z(\omega)\big)
\mu(d\omega).
\end{equation}
Therefore, the result is obtained by applying 
Lemma~\ref{lkIjiiIus3W4-17} with $\mathsf{K}=\RR\oplus\mathsf{G}$
and $\KK=\HH\oplus\GG$.
\end{proof}

\begin{remark}
\label{rkIjiiIus3W9-19}
Proposition~\ref{pkIjiiIus3W9-19} provides a general setting for
computing the proximity operators of abstract integral functionals
by reducing it to the computation of the proximity operator of
the integrand. In particular,
by suitably choosing the underlying measure space and the
integrand, it provides a framework for computing the proximity
operators of the integral function based on perspective functions
discussed in \cite{Scda16a}, which include general divergences.
For instance, discrete $N$-dimensional divergences are obtained by
setting
$\Omega=\{1,\ldots,N\}$ and $\mathcal{F}=2^\Omega$, and letting
$\mu$ be the counting measure (hence $\HH=\GG=\RR^N$) and
$\mathsf{G}=\RR$. While completing the present paper, it has come
to our attention that the computation of the proximity operators of
discrete divergences has also been recently addressed in
\cite{Pesq16}. 
\end{remark}

\subsection{Further results}
\label{subsec:further}
A convenient assumption in 
Theorem~\ref{tevU74gdnx-06}\ref{tevU74gdnx-06ii} is that
$\dom\varphi^*$ is open, as it allowed us to rule out the
case when 
\begin{equation}
\label{eQrs54tGhfnnUjH0-09a}
\prox_{\gamma\widetilde{\varphi}}(\eta,y)=(0,q)
\quad\text{and}\quad q\neq 0,
\end{equation}
and to reduce \eqref{e:hj4mB} to 
\eqref{eevU74gdnx-12b} using \eqref{ekIjiiIus3W9-15}.
In general, \eqref{ekIjiiIus3W9-15} has the form 
$N_{\dom h}(\pi,p)=\{0\}\times N_{\dom\varphi^*}p$ and, if 
$\dom\varphi^*$ is simple enough, explicit expressions can still be
obtained. To shed more light on the case \eqref{eQrs54tGhfnnUjH0-09a}, 
consider the scenario in which $q\neq 0$ and 
$\dom\varphi^*$ is closed, and 
set $p=(y-q)/\gamma$. Then, in view of \eqref{e:jjm}, 
\eqref{eQrs54tGhfnnUjH0-09a}
yields $(\eta/\gamma,p)\in\partial\widetilde{\varphi}(0,q)$. In
turn, we derive from \eqref{ekIjiiIus3W9-23x} that
\begin{equation}
\label{eQrs54tGhfnnUjH0-09b}
\varphi^*(p)\leq -\eta/\gamma\quad\text{and}\quad
\sigma_{\dom\varphi^*}(q)=\scal{p}{q}.
\end{equation}
Thus, 
\begin{equation}
\label{eQrs54tGhfnnUjH0-09c}
p\in {\dom\varphi^*}\quad\text{and}\quad
(\forall z\in {\dom\varphi^*})\quad
\scal{z-p}{y/\gamma-p}\leq 0,
\end{equation}
and we infer from \eqref{e:kolmogorov} that 
$p=P_{\dom\varphi^*}(y/\eta)$. Therefore, 
\begin{equation}
\label{eQrs54tGhfnnUjH0-09e}
\prox_{\gamma\widetilde{\varphi}}(\eta,y)=
\big(0,y-\gamma P_{\dom\varphi^*}(y/\eta)\big)=
\big(0,y-P_{\gamma\dom\varphi^*}y\big)
\end{equation}
and we note that the condition $q\neq 0$ means that
$y\notin\gamma\,{\dom\varphi^*}$. 
We provide below examples in which $\dom\varphi^*$ is a simple
proper closed subset of $\GG$ and the proximity operator of the
perspective function of $\varphi$ can be computed explicitly.

\begin{example}
\label{exQrs54tGhfnnUjH0-01}
Suppose that $D\neq\{0\}$ is a nonempty closed convex cone in $\GG$ 
and define 
\begin{equation}
\label{eQrs54tGhfnnUjH0-03p}
\varphi=\vartheta+\iota_D,\quad\text{where}\quad
\vartheta=\sqrt{1+\|\cdot\|_{\GG}^2}.
\end{equation}
Since $\dom\vartheta=\GG$, we have $\varphi^*=(\vartheta+\iota_D)^*
=\vartheta^*\infconv\iota_{D^\ominus}$, where $D^\ominus$ is the
polar cone of $D$ and (combine 
\cite[Examples~13.2(vi) and 13.7]{Livre1})
\begin{equation}
\label{eQrs54tGhfnnUjH0-01b}
\vartheta^*\colon\GG\to\RX\colon u\mapsto
\begin{cases}
-\sqrt{1-\|u\|_\GG^2},&\text{if}\;\;\|u\|_\GG\leq 1;\\
\pinf,&\text{if}\;\;\|u\|_\GG>1.
\end{cases}
\end{equation}
Thus, 
$\dom\varphi^*=\dom(\vartheta^*\infconv\iota_{D^\ominus})
=\dom\vartheta^*+\dom\iota_{D^\ominus}=B(0;1)+D^\ominus$ is
closed as the sum of two closed convex sets, one of which is
bounded. As a result, since $D^\ominus\neq\GG$, 
\begin{equation}
\label{eQrs54tGhfnnUjH0-03w}
\dom\varphi^*\;\text{is a proper closed subset of}\;\GG.
\end{equation}
Now set $\KK=\RR\oplus\GG$ and $K=\RP\times D$, and let
$\gamma\in\RPP$, $\eta\in\RR$, and $y\in\GG$. Then
$\|(\eta,y)\|_\KK=\sqrt{|\eta|^2+\|y\|_\GG^2}$
and, as shown in \cite[Example~3.5]{Scda16a}, 
\begin{equation}
\label{eQrs54tGhfnnUjH0-03u}
\widetilde{\varphi}=\|\cdot\|_\KK+\iota_K.
\end{equation}
Hence, we derive from \eqref{eQrs54tGhfnnUjH0-03g} that 
\begin{equation}
\label{eQrs54tGhfnnUjH0-03x}
\prox_{\gamma\widetilde{\varphi}}(\eta,y)=
\begin{cases}
(0,0),&\text{if}\;\;\|P_{K}(\eta,y)\|_\KK\leq\gamma;\\[3mm]
\bigg(1-\dfrac{\gamma}{\|P_{K}(\eta,y)\|_\KK}\bigg)
P_{K}(\eta,y),&\text{if}\;\;\|P_{K}(\eta,y)\|_\KK>\gamma.
\end{cases}
\end{equation}
We thus obtain an explicit expression as soon as $P_K$ is explicit
although $\dom\varphi^*$ is not open. As an illustration, 
let $N\geq 2$ be an integer, set $\GG=\RR^{N-1}$, let 
$D=\RP^{N-1}$, and denote by $\|\cdot\|_{N}$ the usual 
$N$-dimensional
Euclidean norm. Then $\varphi=\sqrt{1+\|\cdot\|_{N-1}^2}+\iota_D$,
$K=\RP^N$, and \eqref{eQrs54tGhfnnUjH0-03x} becomes 
\begin{align}
\label{eQrs54tGhfnnUjH0-01e}
\prox_{\gamma\widetilde{\varphi}}(\eta,y)=
\begin{cases}
(0,0),&\text{if}\;\;\|(\eta_+,y_+)\|_N\leq\gamma;\\[3mm]
\bigg(1-\dfrac{\gamma}{\|(\eta_+,y_+)\|_N}\bigg)(\eta_+,y_+),
&\text{if}\;\;\|(\eta_+,y_+)\|_N>\gamma,
\end{cases}
\end{align}
where $\eta_+=\max\{0,\eta\}$ and $y_+$ is defined likewise
componentwise.
\end{example}

The second example provides the proximity operator of the
perspective function of the Huber function.

\begin{example}[perspective of the Huber function]
\label{exQrs54tGhfnnUjH0-09}
Following \cite[Example~3.2]{Scda16a}, let $\rho\in\RPP$ and 
consider the perspective function 
\begin{equation}
\label{ekIjiiIus3W9-24c}
\widetilde{\varphi}\colon\RR^2\to\RX\colon (\eta,y)\mapsto
\begin{cases}
\rho|y|-\dfrac{\eta\rho^2}{2},&\text{if}\;\;|y|>\eta\rho\;\;
\text{and}\;\;\eta>0;\\
\dfrac{|y|^2}{2\eta},&\text{if}\;\;|y|\leq\eta\rho\;\;
\text{and}\;\;\eta>0;\\
\rho|y|,&\text{if}\;\;\eta=0;\\
\pinf,&\text{if}\;\;\eta<0
\end{cases}
\end{equation}
of the Huber function 
\begin{equation}
\label{e:rodos2015-08-12}
\varphi\colon\RR\to\RX\colon y\mapsto
\begin{cases}
\rho|y|-\Frac{\rho^2}{2}, &\text{if}\;\;|y|>\rho;\\[+2mm]
\Frac{|y|^2}{2},&\text{if}\;\;|y|\leq\rho.
\end{cases}
\end{equation}
Then $\varphi^*=|\cdot|^2/2+\iota_{[-\rho,\rho]}$ and 
$\dom\varphi^*$ is therefore a proper closed subset of $\RR$.
In addition, \eqref{eqerY76t0k3-10A} yields
\begin{equation}
C=\menge{(\mu,u)\in\RM\times[-\rho,\rho]}{\mu+|u|^2/2\leq 0}.
\end{equation}
Now let $\eta\in\RR$, let $y\in\RR$, and set
$(\chi,q)=\prox_{\gamma\widetilde{\varphi}}(\eta,y)$.  
Then the following hold:
\begin{enumerate}
\item
If $\eta+|y|^2/(2\gamma)\leq 0$ and $|y|\leq\gamma\rho$, then 
Theorem~\ref{tevU74gdnx-06}\ref{tevU74gdnx-06i} yields
$(\chi,q)=(0,0)$.
\item
We have $\chi=0$ $\Leftrightarrow$ $\eta/\gamma\leq-\rho^2/2$. 
Hence, if $\eta\leq-\gamma\rho^2/2$ and $|y|>\gamma\rho$, 
\eqref{eQrs54tGhfnnUjH0-09e} yields
$(\chi,q)=(0,y-P_{[-\gamma\rho,\gamma\rho]}y)=
(0,y-\gamma\rho\,\sign(y))$.
\item
If $\eta>-\gamma\rho^2/2$ and $|y|>\rho\eta+\gamma\rho(1+\rho^2/2)$, 
then $(\eta/\gamma,y/\gamma)\in(-\rho^2/2,\rho\,\sign(y))+
N_C(-\rho^2/2,\rho\,\sign(y))$ and therefore 
$P_C(\eta/\gamma,y/\gamma)=(-\rho^2/2,\rho\,\sign(y))$. Hence, 
\eqref{eqerY76t0k3-10b} yields
$(\chi,q)=(\eta+\gamma\rho^2/2,y-\gamma\rho\,\sign(y))$.
\item
If $\eta>-\gamma\rho^2/2$ and 
$|y|\leq\rho\eta+\gamma\rho(1+\rho^2/2)$, 
then $(\chi,q)=\prox_{\gamma[|\cdot|^2/2]^\sim}(\eta,y)$ 
is obtained by setting $v=0$, $\delta=0$, and $\alpha=2$ in 
Example~\ref{exevU74gdnx-12a}.  
\end{enumerate}
\end{example}

The last example concerns the Vapnik loss function.
\begin{example}[perspective of the Vapnik function]
\label{exQrs54tGhfnnUjH0-11}
Following \cite[Example~3.4]{Scda16a}, let $\varepsilon\in\RPP$ and 
consider the perspective function 
\begin{equation}
\label{eQrs54tGhfnnUjH0-11c}
\widetilde{\varphi}\colon\RR^2\to\RX\colon(\eta,y)\mapsto
\begin{cases}
d_{[-\varepsilon\eta,\varepsilon\eta]}(y),&\text{if}\;\;
\eta\geq 0;\\
\pinf,&\text{if}\;\;\eta<0
\end{cases}
\end{equation}
of the Vapnik $\varepsilon$-insensitive loss function \cite{Vapn00}
\begin{equation}
\label{eQrs54tGhfnnUjH0-11d}
\varphi=\text{max}\{|\cdot|-\varepsilon,0\}.
\end{equation}
We have
$\varphi=d_{[-\varepsilon,\varepsilon]}=
\iota_{[-\varepsilon,\varepsilon]}
\infconv|\cdot|$ 
and therefore $\varphi^*=\varepsilon|\cdot|+\iota_{[-1,1]}$.
Furthermore, \eqref{eqerY76t0k3-10A} becomes
\begin{equation}
C=\menge{(\mu,u)\in\RM\times[-1,1]}{\mu+\varepsilon|u|\leq 0}.
\end{equation}
Now let $\eta\in\RR$, let $y\in\RR$, and set
$(\chi,q)=\prox_{\gamma\widetilde{\varphi}}(\eta,y)$.  
Then the following hold:
\begin{enumerate}
\item
If $\eta+\varepsilon|y|\leq 0$ and $|y|\leq\gamma$, then 
Theorem~\ref{tevU74gdnx-06}\ref{tevU74gdnx-06i} yields
$(\chi,q)=(0,0)$.
\item
We have $\chi=0$ $\Leftrightarrow$ $\eta/\gamma\leq-\varepsilon$. 
Hence, if $\eta\leq-\gamma\varepsilon$ and $|y|>\gamma$, 
\eqref{eQrs54tGhfnnUjH0-09e} yields
$(\chi,q)=(0,y-P_{[-\gamma,\gamma]}y)=
(0,y-\gamma\,\sign(y))$.
\item
If $\eta>-\gamma\varepsilon$ and
$|y|>\varepsilon\eta+\gamma(1+\varepsilon^2)$, then
$(\eta/\gamma,y/\gamma)\in(-\varepsilon,\sign(y))+
N_C(-\varepsilon,\sign(y))$ and therefore 
$P_C(\eta/\gamma,y/\gamma)=(-\varepsilon,\sign(y))$. 
Hence, \eqref{eqerY76t0k3-10b} yields
$(\chi,q)=(\eta+\gamma\varepsilon,y-\gamma\,\sign(y))$.
\item
If $|y|>-\eta/\varepsilon$ and
$\varepsilon\eta\leq|y|\leq\varepsilon\eta+\gamma(1+\varepsilon^2)$, 
then $P_C(\eta/\gamma,y/\gamma)$ coincides with the projection of
$(\eta/\gamma,y/\gamma)$ onto the half-space with outer normal 
vector $(1,\varepsilon\,\sign(y))$ and which has the origin on
its boundary. As a result, 
\eqref{eqerY76t0k3-10b} yields
$(\chi,q)=((\eta+\varepsilon|y|)/(1+\varepsilon^2),
\varepsilon(\eta+\varepsilon|y|)\sign(y)/(1+\varepsilon^2))$.
\item
If $\eta\geq 0$ and
$|y|\leq\varepsilon\eta$, then
$P_C(\eta/\gamma,y/\gamma)=(0,0)$ and 
\eqref{eqerY76t0k3-10b} yields $(\chi,q)=(\eta,y)$.
\end{enumerate}
\end{example}

\section{Applications in high-dimensional statistics}
\label{sec:4}
Sections~\ref{sec:2} and \ref{sec:3}
provide a unifying framework to model a variety of problems around
the notion of a perspective function. By applying the results of 
Section~\ref{sec:3} in existing proximal algorithms, we obtain
efficient methods to solve complex problems. To illustrate this
point, we focus on a specific application area: high-dimensional 
regression in the statistical linear model. 

\subsection{Penalized linear regression}
We consider the standard statistical linear model
\begin{equation}
\label{e:linmodel}
z=Xb+\sigma e,
\end{equation}
where $z=(\zeta_i)_{1\leq i\leq n}\in\RR^n$ is the response, 
$X\in\RR^{n \times p}$ a
design (or feature) matrix, $b=(\beta_j)_{1\leq j\leq p}\in\RR^p$ 
a vector of regression coefficients, $\sigma\in\RPP$, 
and $e=(\varepsilon_i)_{1\leq i\leq n}$ the noise vector; 
each $\varepsilon_i$ 
is the realization of a random variable with 
mean zero and variance $1$. Henceforth, we denote by $X_{i:}$ the
$i$th row of $X$ and by $X_{:j}$ the $j$th column of $X$.  In the
high-dimensional setting where $p > n$, a typical assumption
about the regression vector $b$ is sparsity. In this scenario,
the Lasso \cite{Tibshirani1996} has become a fundamental tool for
variable
selection and predictive modeling. It is based on solving
the penalized least-squares problem
\begin{equation}
\label{e:lasso}
\minimize{b\in\RR^p}{\frac{1}{2n}\|Xb-z\|^2_2+\lambda\|b\|_1},
\end{equation}
where $\lambda\in\RP$ is a regularization parameter that aims at
controlling the sparsity of the solution. The Lasso has strong 
performance guarantees in terms of support recovery, estimation, 
and predictive performance if one takes 
$\lambda\propto\sigma\|X^\top e\|_\infty$. 
In the high-dimensional setting, two shortcomings of the Lasso are
the introduction of bias in the final estimates due to the $\ell^1$
norm and lack of knowledge about the quantity $\sigma$ which
necessitates proper tuning of $\lambda$ via model selection
strategies that is dependent on $\sigma$. Bias reduction can be
achieved by using a properly weighted $\ell^1$ norm, resulting in
the adaptive Lasso \cite{Zhou2006} formulation
\begin{equation}
\label{e:alasso}
\minimize{b\in\RR^p}{\frac{1}{2n}\|Xb-z\|^2_2+
\lambda\sum_{j=1}^{p}w_j|\beta_j|},
\end{equation}
where the fixed weights $w_j\in\RPP$ are estimated from data. In
\cite{Zhou2006}, it was shown that, for suitable choices of $w_j$,
the adaptive Lasso produces (asymptotically) unbiased estimates of
$b$. One of the first methods to alleviate the $\sigma$-dependency
of the Lasso has been the Sqrt-Lasso \cite{Belloni2011}. The 
Sqrt-Lasso problem is based on the formulation
\begin{equation}
\label{e:sqrtlasso}
\minimize{b\in\RR^p}{\frac{1}{2}\|Xb-z\|_2+
\lambda\|b\|_1}.
\end{equation}
This optimization problem can be cast as second order 
cone program (SOCP) \cite{Belloni2011}. The modification of the 
objective function can be interpreted as an (implicit) scaling 
of the Lasso objective function by an estimate 
$\|X b-z\|_2/\sqrt{n}$ of $\sigma$ \cite{Lederer2015}, leading to
\begin{equation}
\label{e:sqrtlasso2}
\minimize{b\in\RR^p}{\frac{1}{2\sqrt{n}}\frac{\|Xb-z\|^2_2}
{\Frac{1}{\sqrt{n}}\|Xb-z\|_2}+\lambda\|b\|_1}.
\end{equation}
In \cite{Belloni2011}, it was shown that the tuning parameter
$\lambda$ does not depend on $\sigma$ in Sqrt-Lasso. 

Alternative approaches rely on the idea of simultaneously and
explicitly estimating $b$ and $\sigma$ from the data. The scaled
Lasso \cite{Sun2012}, a robust hybrid of ridge and Lasso regression
\cite{Owen07}, and the TREX \cite{Lederer2015} are important
instances. In the following, we will show that these estimators are
based on perspective functions under the unifying statistical
framework of concomitant estimation. We will introduce a novel
family of estimators and show how the corresponding optimization
problems can be solved using proximal algorithms. In particular, we
will derive novel proximal algorithms for solving both the 
standard TREX and a novel generalized version of the TREX 
which includes the Sqrt-Lasso as special case.

\subsection{Penalized concomitant M-estimators}
In statistics, the task of simultaneously estimating a regression
vector $b$ and an additional model parameter is referred to as
concomitant estimation.  In \cite{Huber1981}, Huber introduced a
generic method for formulating ``maximum likelihood-type"
estimators (or M-estimators) with a concomitant parameter from a
convex criterion. Using our perspective function framework, we can
extend this framework and introduce the class of penalized
concomitant M-estimators defined through the convex optimization
problem 
\begin{equation}
\label{e:concom}
\minimize{\sigma\in\RR,\,\tau\in\RR,\,b\in\RR^p}
{\sum_{i=1}^n\widetilde{\varphi_i}\big(\sigma,X_{i:}b-\zeta_i\big)
+\sum_{j=1}^p\widetilde{\psi}_j\big(\tau,a_j^\top b\big)},
\end{equation}
with concomitant variables $\sigma$ and $\tau$ under the
assumptions outlined in Theorem~\ref{tevU74gdnx-06} and in
Section~\ref{subsec:further}. Here,
$\varphi_i\in\Gamma_0(\RR)$,
$\psi_j\in\Gamma_0(\RR)$, and $a_j\in\RR^p$.
The terms
$\widetilde{\varphi_i}$ are data fitting terms and
$\widetilde{\psi}_j$ are penalty terms. A prominent instance
of this family of estimators is the scaled Lasso \cite{Sun2012} 
formulation 
\begin{equation}
\label{e:scaledlasso}
\minimize{b\in\RR^p,\,\sigma\in\RPP}{\frac{1}{2n}\frac{\|Xb-z\|^2_2}
{\sigma}+\frac{\sigma}{2}+\lambda\|b\|_1},
\end{equation}
which yields estimates equivalent to the Sqrt-Lasso. 
Here, setting $\varphi_i=|\cdot|^2/(2n)+1/2$ and
$\psi_j=\lambda|\cdot|$ leads to the scaled (or concomitant) Lasso
formulation (see Lemma~\ref{levU74gdnx-27},
Corollary~\ref{cevU74gdnx-29a}, and \cite{Ndia16}). Other
function choices result in well-known estimators. For instance,
taking each ${\varphi_i}$ to be the Huber function (see
Example~\ref{exQrs54tGhfnnUjH0-09}) and each ${\psi_j}$ to be the Berhu
(reversed Huber) function recovers the robust Lasso variant,
introduced and discussed in \cite{Owen07}. Setting each
$\psi_j=\lambda|w_j \cdot|$ to be a weighted $\ell^1$ component
results in the ``Huber + adaptive Lasso" estimator, analyzed
theoretically in \cite{Lamb2011}. Note that for the latter two
approaches, no dedicated optimization algorithms exist that can
solve the corresponding optimization problem with provable
convergence guarantees. Combining the proximity operators
introduced here with proximal algorithms enables us to design
such algorithms. To exemplify this powerful framework we focus
next on a particular instance of a penalized concomitant
M-estimator, the TREX estimator, and derive proximity operators and
proximal algorithms. 

\subsection{Proximal algorithms for the TREX}
The TREX \cite{Lederer2015} extends Sqrt-Lasso and scaled Lasso by
taking into account
the unknown noise distribution of $e$. Recalling that a
theoretically desirable tuning parameter for the Lasso is
$\lambda\propto\sigma\|X^\top e\|_\infty$, the TREX scales
the Lasso objective by an estimate of this quantity, namely,
\begin{align}
\label{eq:trex}
\minimize{b\in\RR^p}{\frac{\|Xb-z\|_2^2}{\|X^\top(Xb-z)\|_\infty}
+\alpha\|b\|_1}.
\end{align}
The parameter $\alpha>0$ can be set to a constant value
($\alpha=1/2$ being the default choice). In \cite{Lederer2015},
promising statistical results were reported where an approximate
version of the TREX, with no tuning of $\alpha$, has been shown to
be a valid alternative to the Lasso. A major technical challenge in
the TREX formulation is the non-convexity of the optimization
problem.  In \cite{Bien2016}, this difficulty is overcome by
showing that the TREX problem, although non-convex, can be solved
by observing that problem \eqref{eq:trex} can be equivalently
expressed as finding the best solution to $2p$ convex problems 
of the form
\begin{equation}
\label{eq:convexTREX}
\minimize{\substack{b\in\RR^p\\ x_j^\top(X b-z)>0}}
{\frac{\|Xb-z\|^2_2}{\alpha x_j^\top(X b-z)}+\|b\|_1},
\quad\text{where}\quad
x_j=s X_{:j},\quad\text{with}\quad s\in\{-1,1\}. 
\end{equation}
Each subproblem can be reformulated as a
standard SOCP and numerically solved using generic SOCP solvers
\cite{Bien2016}. 
Next we show how our perspective function approach
allows us to derive proximal algorithms for not only the TREX
subproblems and but also for novel generalized versions of the
TREX. The proximal algorithms construct a sequence
$(b_k)_{k\in\NN}$ that is guaranteed to converge to a solution to
\eqref{eq:convexTREX}.

\subsubsection{Proximal operators for the TREX subproblem}
\label{subsec:proxTREX}
We first note that the data fitting term of the TREX subproblem 
\eqref{eq:convexTREX} is the special case of \eqref{ekIjiiIus3W9-24t} 
where $\HH=\RR^p$, $\GG=\RR^n$, $q=2$, 
$L=X$, $r=z$, $u=X^\top x_j$, and $\rho=x_j^\top z$. Given
$\alpha\in\RPP$, the data fitting term of the TREX subproblem thus
assumes the form
\begin{equation}
\label{e:lscperspective3}
f_j\colon\RR^p\to\RX\colon b\mapsto
\begin{cases}
\Frac{\|X b-z\|_2^2}{\alpha x_j^\top(X b-z)},
&\text{if}\;\;x_j^\top(X b-z)>0;\\
0,&\text{if}\;\;X b=z;\\
\pinf,&\text{otherwise,}
\end{cases}
\end{equation}
and the corresponding TREX subproblem is to
\begin{equation}
\label{exevU74gdnx-02b}
\minimize{b\in\RR^p}{f_j(b)+\|b\|_1} \,.
\end{equation}
Now consider the linear transformation
\begin{equation}
\label{exevU74gdnx-02a}
M_j\colon\RR^p\to\RR\times\RR^n\colon b\mapsto
\big(x_j^\top X b,X b\big)
\end{equation}
and introduce
\begin{equation}
\label{exevU74gdnx-02c}
g_j\colon\RR\times\RR^n\to\RX\colon(\eta,y)\mapsto
\begin{cases}
\Frac{\|y-z\|_2^2}{\alpha\big(\eta-x_j^\top z\big)},
&\text{if}\;\;\eta>x_j^\top z; \\
0,&\text{if}\;\;y=z\;\text{and}\;\eta=x_j^\top z;\\
\pinf,&\text{otherwise.}
\end{cases}
\end{equation}
Then $f_j=g_j\circ M_j$. Upon setting $h=\|\cdot\|_1$, we see that
\eqref{exevU74gdnx-02b} is of the form 
\begin{equation}
\label{exevU74gdnx-02c2}
\minimize{b\in\RR^p}{g_j(M_j b)+h( b)}.
\end{equation}
Next, we determine the proximity operators $\prox_{g_j}$ 
and $\prox_h$, as only those are needed in modern proximal
splitting methods \cite{MaPr16,Siop15} to solve 
\eqref{exevU74gdnx-02c2}.
The proximity operator $\prox_h$ is the standard soft thresholding 
operator. A formula for
$\prox_{g_j}$ is provided by Example~\ref{exevU74gdnx-12a} up to
a shift by $(x_j^\top z,z)$. Let $\gamma\in\RPP$ and let $g$ be as
in \eqref{eevU74gdnx-12x}.
Combining Example~\ref{exevU74gdnx-12a} and 
\cite[Proposition~23.29(ii)]{Livre1}, we obtain, for every
$\eta\in\RR$ and every $y\in\RR^n$, 
\begin{align}
\label{exevU74gdnx-02f}
\prox_{\gamma g_j}(\eta,y)
&=(x_j^\top z,z)+\prox_{\gamma g_j}
\big(\eta-x_j^\top z,y-z\big)\nonumber\\
&=\begin{cases}
\big(\eta+\alpha\gamma\|p\|_2^2/4,y-\gamma p\big),
&\text{if}\;\;
4\gamma(\eta-x_j^\top z)+\alpha\|y-z\|_2^2>0;\\
\big(x_j^\top z,z\big),&\text{if}\;\;
4\gamma(\eta-x_j^\top z)+\alpha\|y-z\|_2^2\leq 0,
\end{cases}
\end{align}
where
\begin{equation}
\label{eDACy47Hj-13X}
p=
\begin{cases}
\dfrac{t}{\|y-z\|}(y-z),&\text{if}\;\;y\neq z;\\
0,&\text{if}\;\;y= z,
\end{cases}
\end{equation}
and where $t$ is the unique solution in $\RPP$ to the depressed cubic
equation 
\begin{equation}
\label{eDACy47Hj-13p}
s^{3}+\frac{4\alpha(\eta-x_j^\top z)+8\gamma}{\alpha^2\gamma}s
-\frac{8\|y-z\|}{\alpha^2\gamma}=0.
\end{equation}

\subsubsection{Proximal operators for generalized TREX estimators}
\label{subsec:proxGTREX}
Thus far, we have shown that the data-fitting function in the 
TREX subproblem \eqref{eq:convexTREX} is a special case
of \eqref{ekIjiiIus3W9-24t}. However, the full potential of
\eqref{ekIjiiIus3W9-24t} is revealed by taking a general 
$q\in\left]1,\pinf\right[$, leading to the composite perspective 
function 
\begin{equation}
\label{eevU74gdnx-14a}
f_{j,q}\colon\RR^p\to\RX\colon b\mapsto
\begin{cases}
\Frac{\|X b-z\|_2^q}{\alpha \big|x_j^\top(X b-z)\big|^{q-1}},
&\text{if}\;\;x_j^\top(X b-z)>0;\\
0,&\text{if}\;\;X b=z;\\
\pinf,&\text{otherwise}.
\end{cases}
\end{equation}
This function is the
data fitting term of a generalized TREX subproblem for the
corresponding global generalized TREX objective
\begin{align}
\label{eq:gtrex}
\minimize{b\in\RR^p}{\frac{\|Xb-z\|_2^q}
{\alpha\|X^\top(Xb-z)\|^{q-1}_\infty} + \| b\|_1}.
\end{align}
This objective function provides a novel family of generalized TREX
estimators, parameterized by $q$. The first important observation
is that, in the limiting case $q\to 1$, the  generalized TREX
estimator collapses to the Sqrt-Lasso \eqref{e:sqrtlasso}.
Secondly, particular choices of $q$ allow very efficient
computation of proximity operators for the generalized TREX
subproblems. Considering the linear transformation
$M_j\colon\RR^p\to\RR\times\RR^n\colon b\mapsto
\big(x_j^\top X b,X b\big)$
and introducing 
\begin{equation}
\label{eevU74gdnx-14d}
g_{j,q}\colon\RR\times\RR^n\to\RX\colon(\eta,y)\mapsto
\begin{cases}
\Frac{\|y-z\|_2^q}{\alpha\big|\eta-x_j^\top z\big|^{q-1}},
&\text{if}\;\;\eta>x_j^\top z; \\
0,&\text{if}\;\;y=z\;\text{and}\;\eta=x_j^\top z;\\
\pinf,&\text{otherwise}
\end{cases}
\end{equation}
we arrive at $f_{j,q}=g_{j,q}\circ M_j$. Setting
$h=\|\cdot\|_1$ the corresponding problem is to
\begin{equation}
\label{exk998JuY4-07}
\minimize{b\in\RR^p}{g_{j,q}(M_j b)+h( b)}.
\end{equation}
The proximity operator $\prox_{g_{j,q}}$ is provided by
Example~\ref{exevU74gdnx-10a}, where $\delta=0$ and $v=0$, up 
to a shift by $(x_j^\top z,z)$.
Let $g$ be the function in \eqref{eevU74gdnx-10a} and
let $\gamma\in\RPP$. Set $q^*=q/(q-1)$, set 
$\varrho=(\alpha(1-1/q^*))^{q^*-1}$, and take 
$(\eta,y)\in\RR\times\GG$. 
If $q^*\gamma^{q^*-1}(\eta-x_j^\top z)+\varrho\|y-z\|_2^{q^*}>0$
and $y\neq z$, let $t\in\RPP$ be the unique solution to the
polynomial equation
\begin{equation}
\label{eevU74gdnx-14f}
s^{2q^*-1}+\frac{q^*(\eta-x_j^\top z)}{\gamma\varrho}s^{q^*-1}
+\frac{q^*}{\varrho^2}s-\frac{q^*\|y-z\|}{\gamma\varrho^2}=0.
\end{equation}
Set
\begin{equation}
\label{eevU74gdnx-14x}
p=
\begin{cases}
\dfrac{t}{\|y-z\|}(y-z),&\text{if}\;\;y\neq z;\\
0,&\text{if}\;\;y= z.\\
\end{cases}
\end{equation}
Then we derive from Example~\ref{exevU74gdnx-10a} that
\begin{equation}
\label{eevU74gdnx-14h}
\prox_{\gamma g_{j,q}}(\eta,y)=
\begin{cases}
\big(\eta+\gamma\varrho t^{q^*}/q^*,y-\gamma p\big),
&\text{if}\;\;q^*\gamma^{q^*-1}(\eta-x_j^\top z)
+\varrho\|y-z\|_2^{q^*}>0;\\
\big(x_j^\top z,z\big),&\text{if}\;\;q^*\gamma^{q^*-1}
(\eta-x_j^\top z)+\varrho\|y-z\|_2^{q^*}\leq 0.
\end{cases}
\end{equation}
The key step in the calculation of the proximity operator is to
solve \eqref{eevU74gdnx-14f} efficiently. The solution
is explicit for $q=2$, as discussed in 
Example~\ref{exevU74gdnx-12a}.
For $q=3$, we obtain a quartic equation that can also be solved
explicitly. For $q\in\menge{(i+1)/i}{i\in\NN,\,i\geq 2}$
\eqref{eevU74gdnx-14f} is a polynomial with integer exponents and
is thus amenable to efficient root finding algorithms. For a
general $q$, a one-dimensional line search for convex functions on
a bounded interval needs to be performed.    

\subsubsection{Douglas-Rachford for generalized TREX subproblems}
Problem \eqref{exevU74gdnx-02c2} is a standard composite problem
and can be solved via several proximal splitting methods that
require only the ability to compute $\prox_{g_j}$ and $\prox_{h}$;
see \cite{Siop13} and references therein. For large scale
problems, one could also employ recent algorithms that benefit 
from block-coordinate \cite{Siop15} or asynchronous 
block-iterative implementations \cite{MaPr16},
while still guaranteeing the convergence of their sequence 
$(b_k)_{k\in\NN}$
of iterates to a solution to the problem. 
In this section, we focus on a simple implementation based on 
the Douglas-Rachford splitting method \cite{Livre1} in the context 
of the generalized TREX estimation to illustrate the applicability 
and versatility of the tools presented in Sections~\ref{sec:2} 
and~\ref{sec:3}.

Define $F\colon(b,c)\mapsto h(b)+g_{j,q}(c)$
and $G=\iota_V$, where $V$ is the graph of $M_j$, i.e., 
$V=\menge{(b,c)\in\RR^p\times\RR^{n+1}}{M_jb=c}$. 
Then we can rewrite
\eqref{exevU74gdnx-02c2} as
\begin{equation}
\label{exevU74gdnx-03a}
\minimize{\boldsymbol{x}=(b,c)\in\RR^p\times\RR^{n+1}}
{F(\boldsymbol{x})+G(\boldsymbol{x})}
\end{equation}
Let $\gamma\in\RPP$, let $\boldsymbol{y}_0\in\RR^{p+n+1}$, and let 
$(\mu_k)_{k\in\NN}$ be a sequence
in $\left]0,2\right[$ such that $\inf_{k\in\NN}\mu_k>0$ and 
$\sup_{k\in\NN}\mu_k<2$. The Douglas-Rachford algorithm is
\begin{equation}
\label{e:mer-egee07-15}
\begin{array}{l}
\text{for}\;k=0,1,\ldots\\
\left\lfloor
\begin{array}{l}
\boldsymbol{x}_{k}=\prox_{\gamma{G}}\boldsymbol{y}_k\\
\boldsymbol{z}_{k}=\prox_{\gamma{F}}
(2\boldsymbol{x}_{k}-\boldsymbol{y}_k)\\
\boldsymbol{y}_{k+1}=\boldsymbol{y}_k+
\mu_k(\boldsymbol{z}_{k}-\boldsymbol{x}_{k}).
\end{array}
\right.\\[2mm]
\end{array}
\end{equation}
The sequence $(\boldsymbol{x}_k)_{k\in\NN}$ is guaranteed to 
converge to a solution to \eqref{exevU74gdnx-03a} 
\cite[Corollary~27.4]{Livre1}. Note that
\begin{equation}
\label{e:9hfe6-03e}
\prox_{F}\colon (b,c)\mapsto(\prox_{h}b,\prox_{g_{j,q}}c)
\end{equation}
and, in view of \eqref{e:jjm4}, 
\begin{equation}
\label{e:9hfe6-03f}
\prox_{G}\colon (b,c)\mapsto (v,M_jv),
\quad\text{where}\;\;
v=b-M_j^\top\big(\Id+M_jM_j^\top\big)^{-1}(M_jb-c)
\end{equation}
is the projection operator onto $V$. Hence,  
upon setting $R_j=M_j^\top(\Id+M_jM_j^\top)^{-1}$, 
$\boldsymbol{x}_k=(b_k,c_k)\in\RR^p\times\RR^{n+1}$,
$\boldsymbol{y}_k=(x_k,y_k)\in\RR^p\times\RR^{n+1}$, and
$\boldsymbol{z}_k=(z_k,t_k)\in\RR^p\times\RR^{n+1}$,
we can rewrite \eqref{e:mer-egee07-15} as
\begin{equation}
\label{e:9hfe6-03g}
\begin{array}{l}
\text{for}\;k=0,1,\ldots\\
\left\lfloor
\begin{array}{l}
q_k=M_jx_k-y_k\\
b_{k}=x_k-R_jq_k\\
c_{k}=M_jb_{k}\\
z_{k}=\prox_{\gamma{h}}(2b_k-x_k)\\
t_{k}=\prox_{\gamma{g_{j,q}}}(2c_k-y_k)\\
x_{k+1}=x_k+\mu_k(z_k-b_k)\\
y_{k+1}=y_k+\mu_k(t_k-c_k).
\end{array}
\right.\\[2mm]
\end{array}
\end{equation}
Then $(b_k)_{k\in\NN}$ converges to a solution $b$ to
\eqref{exevU74gdnx-02c2} or \eqref{exk998JuY4-07}. 
Note that the matrix $R_j$ needs to be precomputed only once by
inverting a positive definite symmetric matrix.

\subsection{Numerical illustrations}
We illustrate the convergence behavior of the Douglas-Rachford 
algorithm for TREX problems and the statistical performance of
generalized TREX estimators using numerical experiments. All
presented algorithms and experimental evaluations are implemented
in MATLAB and are available at 
\url{http://github.com/muellsen/TREX}. All
algorithms are run in MATLAB 2015a on a MacBook Pro with 2.8 GHz
Intel Core i7 and 16 GB 1600 MHz DDR3 memory. 

\subsubsection{Evaluation of the Douglas-Rachford scheme on 
TREX subproblems}
We first examine the scaling behavior of the Douglas-Rachford
scheme for the TREX subproblem on linear regression tasks. We
simulate synthetic data according to the linear model
\eqref{e:linmodel} with
$m=20$ nonzero variables, regression vector
$b^{*}=[-1,1,-1,\ldots, 0_{p-m}^\top]^\top$, and feature vectors
$X_{i:}\sim N(0,\Sigma)$ with $\Sigma_{ii}=1$ and
$\Sigma_{ij}=0.3$, and Gaussian noise $\varepsilon_i\sim
N(0,\sigma^2)$ with $\sigma=1$. 
Each column $X_{:j}$ is normalized to have norm $\sqrt{n}$. 
We fix the sample size $n=200$ and consider the dimension
$p\in\{20,50,100,200,500,1000,2000\}$. We solve one standard TREX
subproblem (for $s\in\{-1,1\}$, $X_{:1}$, $\alpha=0.5$) over
$d=20$ random realizations of $X$ and $e$. 
For the TREX subproblem 
we consider the proximal Douglas-Rachford algorithm
\ref{e:9hfe6-03g} with parameters $\mu_k\equiv 1.95$ and
$\gamma=70$. We declare that the
Douglas-Rachford algorithm has converged 
at iteration $K$ if $\min\{\|b_{K+1}-b_{K}\|,\|
y_{K+1}-y_{K}\|\}\leq 10^{-10}$, 
resulting in the final estimate $b_K$.

\begin{figure}%[ht] % was H
\centering
\includegraphics[width=0.75\linewidth]{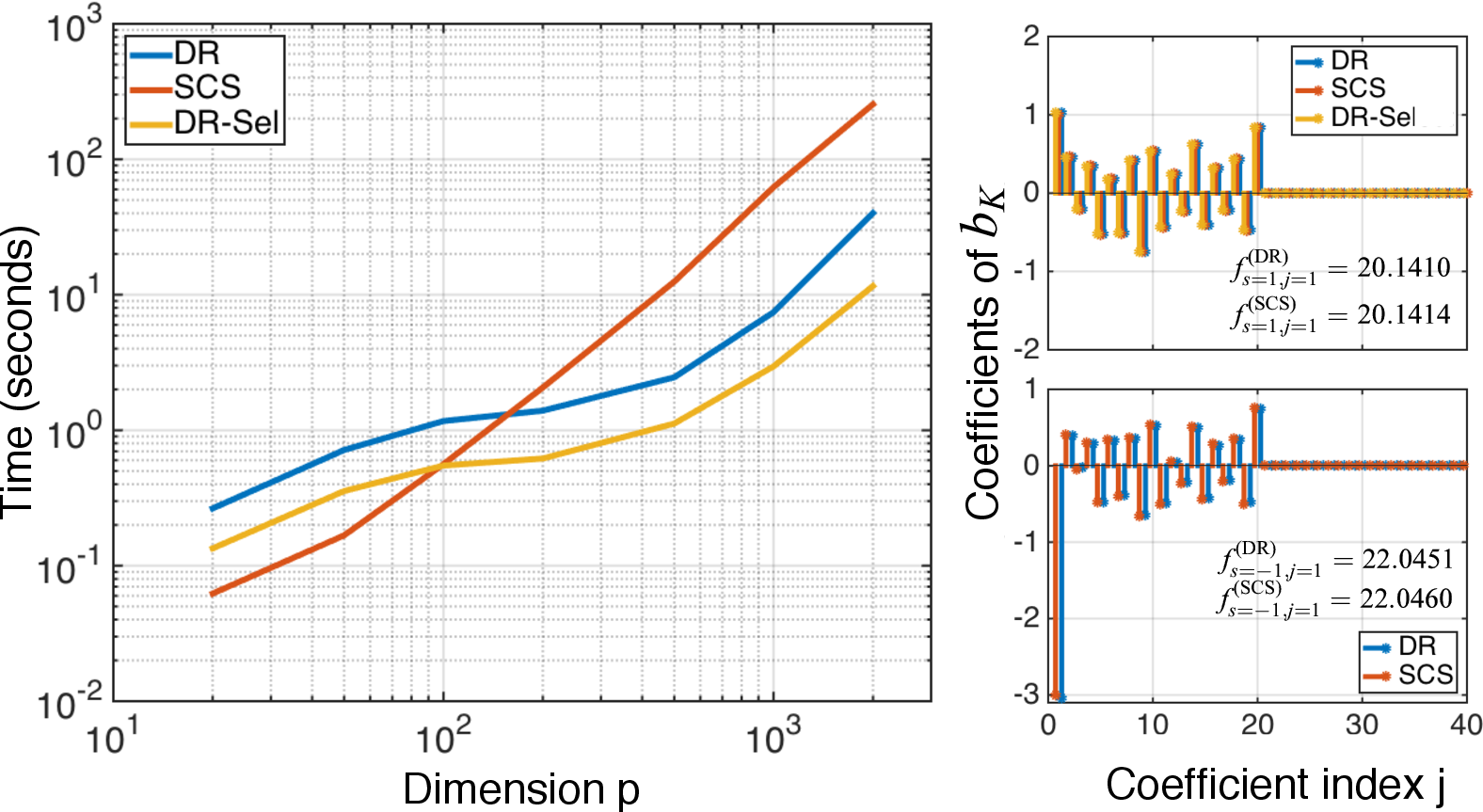}
\caption{Left panel: Average wall-clock time (seconds) versus
dimension $p$ for solving the TREX subproblems with
Douglas-Rachford (DR), SCS, and DR-Sel (Douglas-Rachford with
online sign selection). Right panel: Both plots show the first 40
variables of
a typical $p=2000$ TREX solution (top for $s=+1$). The 
$m=20$ first indices are the non-zero indices in $b^{*}$. Insets
show the
TREX subproblem objective function values for $s={\pm 1}$ and
$X_{:1}$, reached by Douglas-Rachford and SCS. DR-Sel selects the
correct signed subproblem as verified a posteriori by the minimum
function value
($f^\text{(DR)}_{s=1,j=1}=20.1410$ versus
$f^\text{(DR)}_{s=-1,j=1}=22.0451$). 
\label{fig:scaling}}
\end{figure}

In practice, the Douglas-Rachford algorithm for the TREX 
subproblem can be enhanced by an online sign selection rule 
(DR-Sel). When a TREX subproblem for fixed $X_{:j}$ is considered, 
we can solve the problem for $s\in\{-1,1\}$ concurrently for a
small number $k_0$ of iterations
(standard setting $k_0=50$) and select the signed optimization
problem with best progress in terms of objective function value.

We compare the run time scaling and solution quality of
Douglas-Rachford and
DR-Sel with those of the state-of-the-art Splitting Conic Solver
(SCS). SCS is a general-purpose first-order proximal method that
provides numerical solutions to several standard classes of 
optimization problems, including
SOCPs and Semidefinite Programs (SDPs). We use SCS in indirect
mode \cite{ODon2016} to solve the SOCP formulation of the TREX
subproblem \cite{Bien2016} with convergence tolerance
$10^{-4}$. 

The run time scaling results are shown in Figure~\ref{fig:scaling}.
We emphasize that the scaling experiments are not meant to measure
absolute algorithmic performance but rather efficiency with respect
to optimization formulations that are subsequently solved by
proximal algorithms. We observe that SCS with the SOCP
formulation of TREX compares favorably with Douglas-Rachford and
DR-Sel in low dimensions while, for $p>200$, both Douglas-Rachford
variants perform better. DR-Sel outperforms Douglas-Rachford by a
factor of $2$ to $4$ and always selects the correct signed
subproblem (data not shown). The TREX solutions found by SCS and
Douglas-Rachford are close in terms of $\| b^{(DR)}-b^{(SCS)}\|$, 
with DR typically reaching slightly lower function values
than SCS. Values for the first 40 dimensions of a typical 
solution $b_{K}$ in $p=2000$ dimensions are 
shown in Figure~\ref{fig:scaling} (right panels). 

\subsubsection{Behavior of generalized TREX estimators}

We next study the effect of the exponent $q$ on the statistical
behavior of the generalized TREX estimator. We use the synthetic
setting outlined in \cite{Wain2009} to study the phase transition
behavior of the different generalized TREX estimators. We generate
data from the linear model \eqref{e:linmodel} with $p=64$ and 
$m=\lceil 0.4 p^{3/4} \rceil$ nonzero variables, regression vector
$b^*=[-1,1,-1,\ldots, 0_{p-m}^\top]^\top$, and feature vectors
$X_{i:}\sim N(0,\Sigma)$ with $\Sigma_{ii}=1$ and $\Sigma_{ij}=0$ and
Gaussian noise $e$ with $\sigma=0.5$. Each column $X_{:j}$ is
normalized to have norm $\sqrt{n}$. We define the rescaled
sample size according to $ \theta(n, p, m)=n / (2 m \log{(p-m)})$
and consider $\theta(n,p,m)\in\{0.2,0.4,\ldots,1.6\}$.  At
$\theta(n, p,m)=1$, the probability of exact recovery of the
support of $b^*$ is $0.5$ for the (Sqrt)-Lasso with oracle
regularization parameter \cite{Wain2009}. We consider the
generalized TREX with different exponents $q\in\{9/8,7/6,3/2,2\}$
and the Sqrt-Lasso as limiting case $q=1$. For all generalized TREX
estimators we consider regularization parameters
$\alpha\in\{0.1,0.15,\ldots,2\}$. For Sqrt-Lasso we consider the
standard regularization path setting outlined in \cite{Ndia16}. We
solve all generalized TREX problems with the Douglas-Rachford
scheme using the previously described parameter and convergence
settings. We measure the probability of exact support recovery and
Hamming distance to the true support over $d=12$
repetitions. We threshold all ``numerical zeros" in the generalized
TREX solutions vectors at level $0.05$. For all solutions closest
to the true support in terms of Hamming distance, we also 
calculate estimation error $\|b_K-b^*\|_2^2/n$ and 
prediction error $\|Xb_K-Xb^*\|_2^2/n$. 
Figure~\ref{fig:phaseTrans} shows average
performance results across all repetitions. 
\begin{figure}[ht] % was H
\centering
\includegraphics[width=0.9\linewidth]{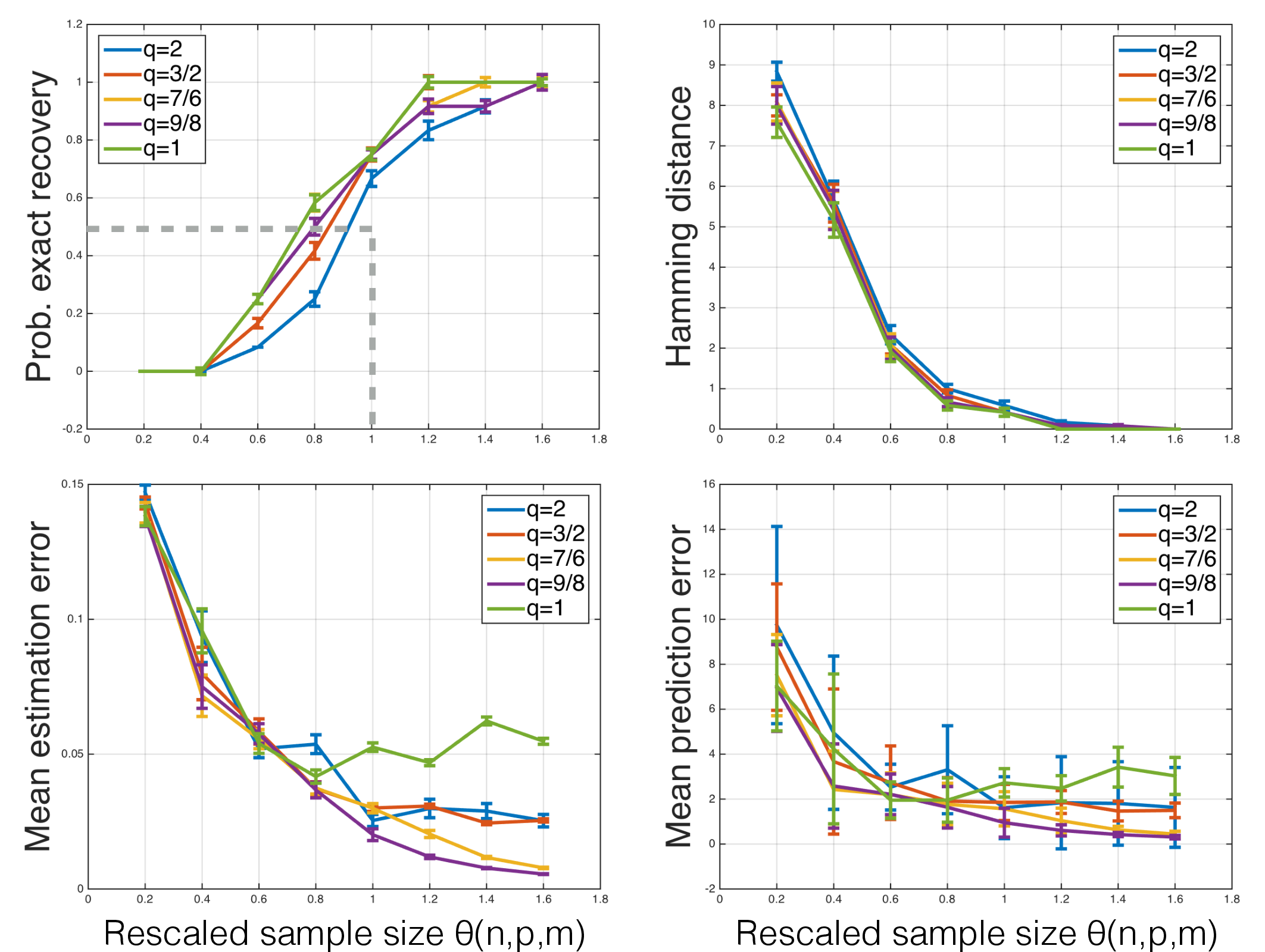}
\caption{Top row: Probability (and standard error) of exact support
recovery versus rescaled sample size $\theta(n, p, m)$ for
generalized TREX with $q\in \{1,9/8,7/6,3/2,2\}$; top right
panel: Average Hamming distance to true support. Bottom row: Mean
estimation error $\|b_K-b^*\|_2^2/n$ (left
panel) and mean prediction error $\|X b_K-X b^*\|_2^2/n$ 
(right panel). 
\label{fig:phaseTrans}}
\end{figure}
We observe several interesting phenomena for the family of
generalized TREX estimators. In terms of exact recovery, the
performance is slightly better than predicted by theory (see gray
dashed line in Figure~\ref{fig:phaseTrans} top left panel), with
decrease in performance for increasing $q$. This is also consistent
with average Hamming distance measurements (top right panel). We
observe that generalized TREX oracle solutions (according to the
minimum Hamming distance criterion) show best performance in terms
of estimation and prediction error for exponents $q\in\{9/8,7/6\}$,
followed by $q\in\{3/2,2\}$. 

The present numerical experiments highlight the usefulness of the
family of generalized TREX estimators for sparse linear regression
problems. Further theoretical research is needed to derive
asymptotic properties of generalized TREX. A central prerequisite
for establishing generalized TREX as statistical estimator is to
solve the underlying optimization problem with provable guarantees.
We have shown that our perspective function framework along with
efficient computation of proximity operators enables this important
task in a seamless way.

\paragraph{Acknowledgement.} We thank Dr. Jacob Bien for valuable
discussions. The Simons Foundation is acknowledged for partial
financial support of this research. The work of P. L. Combettes was
also partially supported by the CNRS MASTODONS project under grant 
2016TABASCO.

\end{document}